\documentclass[a4paper,11pt]{article}
\usepackage{latexsym, amssymb, amsfonts,amsmath}
\usepackage{amssymb,latexsym}
\usepackage{amsfonts}
\usepackage{amsmath,amsthm,graphicx}
\usepackage{bm}
\newcommand{\be}{\begin{equation}}
\newcommand{\ee}{\end{equation}}
\newcommand{\bea}{\begin{eqnarray}}
\newcommand{\eea}{\end{eqnarray}}
\newcommand{\bean}{\begin{eqnarray*}}
\newcommand{\eean}{\end{eqnarray*}}
\newcommand{\brray}{\begin{array}}
\newcommand{\erray}{\end{array}}
\newcommand{\ben}{\begin{equation}{nonumber}}
\newcommand{\een}{\end{equation}{nonumber}}

\newtheorem{dfn}{Definition}[section]
\newtheorem{thm}[dfn]{Theorem}
\newtheorem{lmma}[dfn]{Lemma}

\newtheorem{ppsn}[dfn]{Proposition}
\newtheorem{crlre}[dfn]{Corollary}
\newtheorem{xmpl}[dfn]{Example}
\newtheorem{rmrk}[dfn]{Remark}
\newcommand{\bdfn}{\begin{dfn}}
\newcommand{\bthm}{\begin{thm}}
\newcommand{\otp}{\ot_\gamma}
\newcommand{\ota}{\ot_{alg}}

\newcommand{\blr}{\begin{list}{$($\roman{cnt1}$)$} {\usecounter{cnt1}
        \setlength{\topsep}{0pt} \setlength{\itemsep}{0pt}}}
\newcommand{\bla}{\begin{list}{$($\alph{cnt2}$)$} {\usecounter{cnt2}
       \setlength{\topsep}{0pt} \setlength{\itemsep}{0pt}}}
\newcommand{\bln}{\begin{list}{$($\arabic{cnt3}$)$} {\usecounter{cnt3}
                \setlength{\topsep}{0pt} \setlength{\itemsep}{0pt}}}
\newcommand{\el}{\end{list}}

\newcommand{\blmma}{\begin{lmma}}
\newcommand{\bppsn}{\begin{ppsn}}
\newcommand{\bcrlre}{\begin{crlre}}
\newcommand{\bxmpl}{\begin{xmpl}}
\newcommand{\brmrk}{\begin{rmrk}}
\newcommand{\edfn}{\end{dfn}}
\newcommand{\ethm}{\end{thm}}
\newcommand{\elmma}{\end{lmma}}

\newcommand{\eppsn}{\end{ppsn}}
\newcommand{\ecrlre}{\end{crlre}}
\newcommand{\exmpl}{\end{xmpl}}
\newcommand{\ermrk}{\end{rmrk}}

\newcommand{\IC}{\mathbb{C}}

\newcommand{\IE}{{I\! \! E}}

\newcommand{\IN}{{I\! \! N}}

\newcommand{\IP}{{I\! \! P}}

\newcommand{\IR}{\mathbb{R}}

\newcommand{\IT}{\mathbb{T}}

\newcommand{\IZ}{\mathbb{Z}}

\newcommand{\innerl}{\left\langle}
\newcommand{\innerr}{\right\rangle}

\newcommand{\cla}{{\cal A}}
\newcommand{\clb}{{\cal B}}
\newcommand{\clc}{{\cal C}}
\newcommand{\cld}{{\cal D}}
\newcommand{\cle}{{\cal E}}
\newcommand{\clf}{{\cal F}}
\newcommand{\clg}{{\cal G}}
\newcommand{\clh}{{\cal H}}
\newcommand{\cli}{{\cal I}}
\newcommand{\clj}{{\cal J}}

\newcommand{\cll}{{\cal L}}

\newcommand{\clv}{{\cal V}}
\newcommand{\clw}{{\cal W}}

\newcommand{\cly}{{\cal Y}}

\def\a*{{\cal A}_{h,*}}
\def\B{{\cal B}(h)}
\def\B1{{\cal B}_1(h)}
\def\b{{\cal B}^{\rm s.a.}(h)}
\def\b1{{\cal B}^{\rm s.a.}_1(h)}

\newcommand{\ot}{\otimes}

\newcommand{\lgl}{\langle}
\newcommand{\rgl}{\rangle}

\begin{document}
\begin{center}
 \Large{\bf{A homomorphism theorem and a Trotter product formula for quantum stochastic flows with unbounded coefficients}}\\
 \vspace{0.15in}
{\large Biswarup Das,{\footnote {Indian Statistical Institute, Kolkata;~partially supported by UKIERI.}} }
{\large Debashish Goswami {\footnote{ Indian Statistical Institute, Kolkata; partially supported by Swarnajayanti Fellowship and Grant from DST, Govt. of India.}}}
{\large and Kalyan B. Sinha {\footnote { J.N. Centre for Advanced Scientific Research and Indian Institute of Science, Bangalore .}}}\\
\end{center}
\vspace{0.15in}
\
\begin{abstract}
We give a new method for proving the homomorphic property of a quantum stochastic flow satisfying a quantum stochastic differential equation with unbounded coefficients, under some further hypotheses. As an application, we prove a Trotter product formula for quantum stochastic flows and obtain quantum stochastic dilations of a class of quantum dynamical semigroups generalizing results of \cite{lingaraj} .
\end{abstract}

\section{Introduction}
The theory of quantum stochastic calculus  and of the quantum stochastic differential equations  have been developed over the last three decades (\cite{krp}, \cite{dgkbs}). These methods have many possible applications in the study of non-equilibrium and dissipative physical systems, for example, modeling of damped quantum harmonic oscillator \cite{dgkbs}. They are also of independent mathematical interest, particularly those with unbounded coefficients. 

However, a difficulty in dealing with quantum stochastic calculus with unbounded coefficients is the absence of a convenient method for proving the homomorphic property of a quantum stochastic flow. Suppose that $\mathcal{A}\subseteq B(h)$ is a von-Neumann or C*-algebra and $\mathcal{A}_0\subseteq\mathcal{A}$ is\textit{} a dense (in the appropriate topology) $\ast$-subalgebra, such that there exists a family of completely positive, contractive cocycles $j_t:\mathcal{A}\rightarrow\mathcal{A}^{\prime\prime}\ot B(\Gamma(L^2(\IR_+,k)))~(t\geq0),$ satisfying an equation of the form 
\[
j_t(x)=x+\sum_{\mu,\nu}\int_0^t j_s(\theta^\mu_\nu(x))\Lambda^\nu_\mu(ds), 
\]
where $k$ is a Hilbert space,  $\Gamma(L^2(\IR_+,k))$ is the symmetric Fock space over $L^2(\IR_+,k)$  and $(\theta^\mu_\nu)_{\mu,\nu}$ are linear maps on $\cla_0$. In general, there does not exist a convenient method for showing that such a map $j_t$ is a $\ast$-homomorphism.

Using the quantum It\^{o} formula, one can easily write down algebraic relations which are necessary for a quantum stochastic flow $(j_t)_{t\geq0}$  to be  $\ast$-homomorphic in general. But it is not known whether such algebraic conditions are also sufficient. In this paper, we prove that such algebraic conditions will be sufficient, if we furthermore assume some analytic conditions on the flow as well as the underlying (C* or von-Neumann) algebra. The crucial aspect of this proof is the implementation of an inductive procedure on the It\^{o} formula which is more natural in the set-up of quantum stochastic flows with unbounded coefficients than the usual iterative procedure.

We give a number of applications of the method, including a new Trotter product formula, generalizing the formulas obtained in \cite{lindsaykbs}.

\section{Notation and terminologies.}
\subsection{Some generalities on Hilbert spaces and completely positive maps}
We shall refer the reader to \cite{krp}, \cite{evans}, \cite{dgkbs} and references therein for the basics of the formalism of 
quantum stochastic calculus, which we briefly review here. All the Hilbert spaces appearing in this article will be separable. We adopt the convention that inner-product is linear in the right and conjugate linear in the left. Let $\IC$ and $\IN$ be the set of complex and natural numbers respectively. For $n\in \IN$, $M_n(\IC)$ will denote the set of $n\times n$ matrices with complex entries. For $A\in M_n(\IC)$, we will often write $A\equiv((~a_{ij}~))$, where $\{ a_{ij},i,j=1,\ldots,n\}$ are the entries of $A$. For 
a Hilbert space $\clh$, let $I_\clh$ be the identity operator on $\clh$; write $I_n:=I_{\IC^n}$. We  denote the symmetric Fock space over $\clh$ by $\Gamma(\clh)$ and the exponential vector over $f\in\clh$ by $e(f)$. The vector $e(0)$ is called the vacuum vector. We adopt the convention that for $u\in\clh$, $|u\rgl\in\clh$ and $\lgl u|\in\clh^*$. For $u,v\in\clh$, the rank one operator $|u\rgl\lgl v|\in B(\clh)$ is given by 
$|u\rgl\lgl v|(|w\rgl):=\lgl v,w\rgl |u\rgl$, for all $
|w\rgl\in\clh$. $ Lin(\clv,\clw)$ will denote the space of linear maps from a subspace $D(L)$ (called the domain of $L$) of a vector space $\clv$ to another vector space $\clw$. The tensor product of Hilbert spaces or of operators will usually be denoted by $\ot$, whereas $\ot_{\rm alg}$ will be used for the  algebraic tensor product. We shall also use the projective tensor product $\ot_\gamma$ of Banach spaces, which will be explained later. 

Let $\clh_1 , \clh_2$ be two  Hilbert spaces
and $A\in Lin(\clh_1,\clh _1\ot \clh _2).$ Suppose that $\cld:=D(A)$. For each $f \in \clh  _2$, we define a linear operator $\lgl f, A \rgl$ with domain $\cld$ and  taking values in
$\clh_1$ such that   
\be  \label{2.1}  \lgl  u,\lgl f, A \rgl v \rgl = \lgl
u\ot f, Av\rgl   \ee  for $v \in \cld , ~ u \in \clh _1$.
  We  denote
by $\lgl A, f \rgl$ the adjoint of $\lgl f, A \rgl$,  whenever it exists.

For $\phi\in Lin(D_0\ota\ V_0,\clh_1\ot\clh_2),$ where $D_0$ and $V_0$ are subspaces of $\clh_1$ and $\clh_2$ respectively, $f\in V_0$ and $g\in\clh_2$, define the partial trace of $\phi$ with respect to the rank one operator   $|f><g|$, denoted by $\innerl g,\phi f\innerr:D_0\longrightarrow\clh_1$ by:   
 \[\innerl u,\innerl g, \phi f\innerr v\innerr=\innerl u\ot g),\phi(v\ot f)\innerr,\quad (v\in D_0,u\in\clh_1).\] We also denote by $\phi_f \in Lin(D_0,\clh_1 \ot \clh_2)$ given by $\phi_f(v)=\phi(v \ot f)$.
 
Let $h$ and $\clh$ be Hilbert spaces, $\cla\subseteq B(h)$ be a $C^*$ or von-Neumann algebra, and $\psi_t:\cla\rightarrow B(h\ot\Gamma(L^2(\IR_+,\clh)))$, $t\geq0$ be a family of completely positive, contractive maps. For $c,d\in\clh$, we define a map $\psi^{c,d}_t:\cla\rightarrow B(h)$ as follows:

\be\label{contraction in Fock}
\psi^{c,d}_t(x):=\lgl e(c1_{[0,t)}), \psi_t(x)e(d1_{[0,t)})\rgl\quad(x\in\cla). 
\ee

We conclude this subsection with the following proposition, which will be needed to define a quantum stochastic cocycle.

\begin{ppsn}\label{matrix space tensor}
Let $\clh_1,\clh_2$ and $h$ be Hilbert spaces and $\cla\subseteq B(h)$ be a C*-algebra. Suppose that the maps $\psi_i:\cla\rightarrow\cla^{\prime\prime}\ot B(\clh_i)$, $i=1,2$ are completely positive, contractive and satisfy 
$\lgl\xi_i,\psi_i(x)\eta_i\rgl\in\cla$ for all $x\in\cla$ and
$\xi_i,\eta_i\in\clh_i,$ $i=1,2$. Then there exists a completely positive, contractive map, denoted by $\psi_1\bullet\psi_2$ from
$\cla$ to $\cla^{\prime\prime}\ot B(\clh_1\ot\clh_2)$ satisfying
\begin{enumerate}
\item[(i)]
\[\lgl\xi,(\psi_1\bullet\psi_2)(x)\eta\rgl\in\cla\quad(x\in\cla,\xi,\eta\in\clh_1\ot\clh_2).\]
\item[(ii)]
\[
\lgl\xi_1\ot\xi_2,(\psi_1\bullet\psi_2)(\cdot)(\eta_1\ot\eta_2)\rgl=
\lgl\xi_1,\psi_1(\lgl\xi_2,\psi_2(\cdot)\eta_2\rgl)\eta_1\rgl,
\]
 {\rm for~}  $\xi_i,\eta_i\in\clh_i${\rm ~( $i=1,2$)}.
\end{enumerate}
\end{ppsn}

\begin{proof}
Let $n\geq1$, $\omega:=\sum_{i=1}^{n}u_i\ot\xi_i\ot\eta_i\in h\ot\clh_1\ot\clh_2$, where $u_i\in h$, $\xi_i \in\clh_1$ and $\eta_i\in\clh_2$ for all $i=1,2,\ldots,n$. Note that the map $\IE_{\bm{\underline{\eta}}}:\cla^{\prime\prime}\ot B(\clh_2)\rightarrow M_n(\cla^{\prime\prime})$ ($\cong \cla^{\prime\prime}\ot M_n$) given by $\IE_{\bm{\underline{\eta}}}(X)=((\innerl \eta_i,X\eta_j\innerr))_{i,j=1}^n$ is completely positive. Thus, for $X\geq0$, we have 
\be\label{1 star}
0\leq\IE_{\bm{\underline{\eta}}}(X)\leq\|X\|\IE_{\bm{\underline{\eta}}}(1)=\|X\|1_{\cla^{\prime\prime}}\ot((\innerl\eta_i,\eta_j\innerr)).
\ee
Moreover, if $X(\geq0)$ is such that $\innerl\eta_i,X\eta_j\innerr\in\cla$ for all $i,j$, then $\IE_{\bm{\underline{\eta}}}(X)\in M_n(\cla)$ and 
$Y:=(\psi_1\ot id_{M_n})(\IE_{\bm{\underline{\eta}}}(X))$ is well-defined. As $\psi_1$ is completely positive and contractive, we get the following from \eqref{1 star}:
\[
0\leq (\psi_1 \ot {\rm id}_{M_n})((e_\alpha \ot 1)\IE_{\bm{\underline{\eta}}}(X)(e_\alpha\ot 1))\leq\|X\| 1_{\cla^{\prime\prime}\ot B(\clh_1)}\ot((\innerl\eta_i,\eta_j\innerr)) 
\]
where $\{e_\alpha\}_\alpha$ is any approximate unit for $\cla$. Taking limit with respect to $\alpha$, we get 
\be\label{2 star}
0\leq Y\leq\|X\| 1_{\cla^{\prime\prime}\ot B(\clh_1)}\ot((\innerl\eta_i,\eta_j\innerr)).
\ee
Clearly, $Y\in M_n(\cla^{\prime\prime}\ot B(\clh_1))$ has the matrix form $Y=((Y_{ij} ))$ where $Y_{ij}=\psi_1(\innerl\eta_i,X\eta_j\innerr)$, and by \eqref{2 star} for $u_i\in h$ ($i=1,2,\ldots,n$), 
\[
0\leq\sum_{i,j=1}^n\innerl u_i\ot\xi_i,\psi_1(\innerl\eta_i,X\eta_j\innerr)(u_j\ot\xi_j)\innerr\leq\|X\|\sum_{i,j=1}^n\innerl u_i,u_j\innerr\innerl\xi_i,\xi_j\innerr\innerl\eta_i,\eta_j\innerr. 
\]
As $n,u_i$ and $\eta_i$ are arbitrary, we conclude that there is a unique bounded positive operator $T_X$ (say) in $\cla^{\prime\prime}\ot B(\clh_1)\ot B(\clh_2)$ such that $\|T_X\|\leq\|X\|$ and $\innerl\omega, T_X\omega\innerr=\sum_{i,j=1}^n\innerl u_i\ot\xi_i,\psi_1(\innerl\eta_i,X\eta_j\innerr)(u_j\ot\xi_j)\innerr$ for $\omega=\sum_{i=1}^n u_i\ot\xi_i\ot\eta_i$. Taking $X=\psi_2(x)$ where $x\in\cla$ is positive, we denote by $(\psi_1\bullet\psi_2)(x)$ the operator $T_{\psi_2(x)}$, and then it is easy to extend the map $x\mapsto(\psi_1\bullet\psi_2)(x)$ linearly to all $x\in\cla$. Properties (i) and (ii) in the statement of the proposition now follow easily, thereby completing the proof.
\end{proof}

\begin{rmrk}
In case $\cla$ is a von-Neumann algebra and $\psi_i$ is normal for $i=1,2$, we have 
\[
\psi_1\bullet\psi_2=(\psi_1\ot id_{B(\clh_2)})\circ\psi_2. 
\]
In the more general case treated in \cite{lindsay-wills}, $\psi_1\bullet\psi_2$ coincides with $(\psi_1\ot_M id_{B(\clh_2)})\circ\psi_2$, where $\ot_M$ denotes the matrix space tensor product, introduced in \cite{lindsay-wills}.

\end{rmrk}

\subsection{Quantum stochastic analysis in symmetric Fock space}\label{Fundamental martingales}

For this subsection and rest of the paper, we fix a separable Hilbert space $k_0$, to be called the noise or multiplicity space. Suppose that $\{e_i\}_{i\in\cli}$ is an orthonormal basis for $k_0$, where $\cli=\{ 1,2,\ldots, {\rm dim}(k_0)\}$ if $k_0$ is finite dimensional and $\cli=\IN$ otherwise.  Let $\cle$ denote the subspace spanned by the exponential vectors in $\Gamma(L^2(\IR_+,k_0))$. Set $\Gamma:=\Gamma(L^2(\IR_+,k_0))$, $\Gamma^s_r:=\Gamma(L^2([s,r),k_0))$ ($s<r$), $\Gamma_t:=\Gamma(L^2([0,t),k_0))$ $(t\geq0)$ and $\Gamma^t:=\Gamma(L^2([t,+\infty),k_0))$.

We recall the four fundamental martingales of quantum stochastic calculus (see \cite{krp}):

\begin{enumerate}
\item[Time:]
\[
\{\Lambda^0_0(t)\}_{t\geq0}\in Lin(\cle,\Gamma):\Lambda^0_0(t)e(f):=te(f)\quad(t\geq0,e(f)\in\cle); 
\]
\item[Annihilation:]
\[
\{\Lambda^i_0(t)\}_{t\geq0}\in Lin(\cle,\Gamma):
\Lambda^i_0(t)e(f):=\int_0^t ds f^i(s)e(f)\quad(t\geq0,i\in\cli,e(f)\in\cle),
\]
where $f^i(s):=\lgl e_i,f(s)\rgl$.
\item[Creation:]
\[
\{\Lambda^0_i(t)\}_{t\geq0}\in Lin(\cle,\Gamma):
\Lambda^0_i(t):=\Lambda^i_0(t)^*\quad(t\geq0,i\in\cli)
\]
i.e. for $e(f),e(g)\in\cle$, 
\[
\lgl e(f),\Lambda^0_i(t)e(g)\rgl:=\int_0^t ds f_i(s)\lgl e(f),e(g)\rgl,
\]
where $f_i(s):=\overline{f^i(s)}$.
\item[Number:]
\[
\{\Lambda^i_j(t)\}_{t\geq0}\in Lin(\cle,\Gamma):
\Lambda^i_j(t)e(f):=\int_0^t ds~f^i(s)\Lambda^0_j(s)e(f)\quad(t\geq0,i,j\in\cli,e(f)\in\cle).
\]
\end{enumerate}

The fundamental martingales $\{\Lambda^\alpha_\beta\}_{\alpha,\beta\in\cli\cup\{0\}}$ satisfy the  quantum-It\^o formula (see page 127 of \cite{dgkbs}):

       $$\Lambda^\alpha_\beta(dt)\Lambda^{\alpha^\prime}_{\beta^\prime}(dt)=\hat{\delta}^\alpha_{\beta^\prime} \Lambda^{\alpha^\prime}_\beta(dt)$$ for
 $\alpha,\alpha^\prime,\beta,\beta^\prime=0,1,2,3...$, where 

\be
\begin{split}
\hat{\delta}^\alpha_\beta&:= 0 ~if~ \alpha=0~or~\beta=0\\
&:=\delta_{\alpha,\beta}~{\rm otherwise},
\end{split}
\ee 
$\delta_{\alpha,\beta}$ being the Kr\"onecker delta symbol.
\begin{rmrk}
It follows by a direct computation that 
\[
\lgl e(f),\Lambda^\mu_\nu(t)e(g)\rgl=\lgl \Lambda^\nu_\mu(t)e(f),e(g)\rgl,
\]
or in other words, we have $(\Lambda^\mu_\nu(t))^*=\Lambda^\nu_\mu(t)$, $t\geq0$.
\end{rmrk}

Let $(\theta_t)_{t\geq0}$ denotes the semigroup of unilateral shifts on $L^2(\IR_+,k_0)$, given by:
\[
(\theta_tf)(s)=f(s-t)1_{[t,+\infty)}(s)\quad(t\geq0,s\in\IR_+,f\in L^2(\IR_+,k_0)).
\]
Define $(\Gamma(\theta_t))_{t\geq0}\in B(\Gamma)$ as follows:

\be\label{second quant of time shift}
\Gamma(\theta_t)e(f):=e(\theta_t f)\quad(t\geq0,s\in\IR_+,f\in L^2(\IR_+,k_0)),
\ee
and extending linearly. It can be shown (see p. 169, Section 7.1 of \cite{dgkbs}) that $(\Gamma(\theta_t))_{t\geq0}$ becomes a $C_0$-semigroup of isometries in $\Gamma$. 

Let $\{I_i\}_{i\in\clj}$ be a collection of disjoint sub-intervals of $\IR_+$. It can be shown (page 31, Corollary 2.4.4 of \cite{dgkbs}), that
\begin{equation*}
\Gamma(L^2(\bigcup_i I_i,k_0))\cong\bigotimes_i\Gamma(L^2(I_i,k_0)), 
\end{equation*}
where in the right hand side, the tensor product is with respect to the stabilizing sequence $\{\Omega_{I_i}\}_{i\in\clj}$, $\Omega_{I_i}$ being the vacuum vector in $\Gamma(L^2(I_i,k_0))$, $i\in\clj$. Thus for $r,s\in\IR_+,r<s$, we can view $\Gamma^r_s$ as a subspace of $\Gamma$ via the identification :
\[
\Gamma(L^2([r,s),k_0))\cong\Omega_{[0,r)}\ot\Gamma(L^2([r,s),k_0))\ot\Omega_{[s,+\infty)},
\]
where the Hilbert space on the right hand side is a subspace of $\Gamma_r\ot\Gamma^r_s\ot\Gamma^s$.

Let $h$ be a separable Hilbert space and $\cla\subseteq B(h)$ be a C* or von-Neumann algebra. For $X\in\cla^{\prime\prime}\otimes B(\Gamma_s)$, it can be seen that $(I_h\ot\Gamma(\theta_t))(X\otimes I_{\Gamma^s})(I_h\ot\Gamma(\theta^*_t))$ is of the form $P_{12}(|\Omega_{[0,t)}><\Omega_{[0,t)}|\otimes X_1 \otimes I_{\Gamma^{t+s}})P_{12}^*$, where $X_1 \in\cla^{\prime\prime}\ot B(\Gamma^t_{t+s})$ and $P_{12}:\Gamma_t\otimes h\otimes\Gamma^t\longrightarrow h\otimes\Gamma_t\otimes\Gamma^t$ is the unitary flip between first and second tensor components. Set $\Xi_t(X):=X_1.$

\begin{ppsn}\label{proper xi}

\begin{enumerate}
\item[(I)]
Suppose that $\cla\subset B(h)$ is a C*-algebra. If $X\in\cla^{\prime\prime}\ot B(\Gamma_s)$ such that $\lgl \xi,X\eta\rgl\in\cla$ for all $\xi,\eta\in\Gamma_s$, then 
$\lgl \xi^\prime,\Xi_t(X)\eta^\prime\rgl\in\cla$, for all $\xi^\prime,\eta^\prime\in \Gamma^t_{t+s}$.
\item[(II)]
Suppose that $j_t:\cla\rightarrow\cla^{\prime\prime}\ot B(\Gamma);t\geq0$ is a family of completely positive maps such that 
\[
j_t=\widehat{j}_t\ot I_{\Gamma^t},~
j_t^{c,d}(\cla)\subseteq\cla\quad(t\geq0,c,d\in k_0),\]where $\widehat{j}_t$ is a completely positive map from $\cla$ to $\cla^{\prime\prime}\ot B(\Gamma_t)$. 

Then the map $\clj_{s+t}:= \widehat{j}_s\bullet(\Xi_s\circ \widehat{j}_t)$ is a well-defined map from $\cla$ to $\cla^{\prime\prime}\ot B(\Gamma_s\ot\Gamma^s_{t+s})$.
\end{enumerate} 
\end{ppsn}

\begin{proof}
Let $\xi^\prime,\eta^\prime\in \Gamma^t_{s+t}$. By virtue of the discussions on Fock space and shift operators preceding Proposition \ref{proper xi}, it follows that 
\begin{equation*}
\begin{split}
\innerl\xi^\prime, \Xi_t(X)\eta^\prime\innerr=\innerl\Gamma(\theta_t^*)(\Omega_{[0,t)}\ot\xi^\prime\ot\Omega_{[s+t,+\infty)}),(X\ot I_{\Gamma^s})\Gamma(\theta_t^*)(\Omega_{[0,t)}\ot\eta^\prime\ot\Omega_{[s+t,+\infty)})\innerr.
\end{split}
\end{equation*}
The right hand side of the above equation belongs to $\cla$, which proves (I). (II) follows by combining (I) and Proposition \ref{matrix space tensor}.
\end{proof}

\bdfn\label{cocycle definition}
 A family of contractive and completely positive maps $(j_t)_{t \geq 0}$ from $\cla$ to $\cla^{\prime \prime} \ot B(\Gamma)$ is called a quantum stochastic cocycle on $\cla$ if the family $(j_t)_{t\geq0}$ satisfies:

 \begin{enumerate}
 \item[(i)]{\rm ~Adaptedness:}
\[
j_{t}=\widehat{j}_t\ot I_{\Gamma^t},~j_t^{c,d}(x)\in\cla\quad(x\in\cla,c,d\in k_0), 
\]
where $\widehat{j}_t$ is a completely positive map from $\cla$ to $\cla^{\prime\prime}\ot B(\Gamma_t)$. 
\item[(ii)]{\rm ~Cocycle~property:\textbf{}}
\[
\widehat{j}_{t+s}(X)=\{\widehat{j}_s\bullet(\Xi_s\circ \widehat{j}_t)\}(X)\quad(X\in\cla).
\]
\end{enumerate} 
\edfn  

\begin{rmrk}\label{vacuum expectation semigroup}
 Note that by virtue of Proposition \ref{proper xi}, the right hand side of (ii) is well-defined. 
It also follows from (ii) of Definition \ref{cocycle definition} that $\{ j_t^{c,d}\}_{t \geq 0}$ is a semigroup of maps on $\cla$ (see \cite{accardi}). 
\end{rmrk}

If the C*-algebra $\cla$ is not closed in the ultraweak topology of $B(h)$, then we will refer to the norm topology on $\cla$ as its natural topology. On the other hand, if $\cla$ is a von-Neumann algebra, the ultraweak topology will be chosen as the natural topology. We call a semigroup $(T_t)_{t\geq0}$ of bounded 	maps on $\cla$, a $C_0$-semigroup if the map $t\rightarrow T_t(x)$ is continuous in the natural topology of $\cla$ for every fixed $x\in\cla$. 

We now introduce a special class of quantum stochastic cocycles which are of interest to us in this paper.

\bdfn\label{quantum stochastic flow definition}
A quantum stochastic cocycle $(j_t)_{t\geq0}$ on a C*-algebra $\cla$ is called a (C*-algebraic)  quantum stochastic flow on $\cla$ with the domain algebra $\cla_0$, structure maps $\{\theta^\mu_\nu\}_{\mu,\nu}\in Lin(\cla,\cla)$ and noise space $k_0$, where $\cla_0$ is a $\ast$-subalgebra of $\cla$, $\mu,\nu\in\{0\}\cup \cli$, if the following hold :

\begin{enumerate}

\item[(i)]
The map $t\rightarrow j_t(x)$ is weakly measurable for every $x\in\cla$, in the sense that
\[
t\mapsto \lgl u\ot\xi,j_t(x)(v\ot\eta)\rgl\quad(u,v\in h,\xi,\eta\in\Gamma)
\]
is a Borel-measurable map from $\IR_+$ to $\IC$.

\item[(ii)]
The domain algebra $\cla_0$ is norm-dense in $\cla$, $\cla_0\subseteq D(\theta^\mu_\nu),\mu,\nu\geq0$, and the family $\{j_t(x)\}_{t\geq0}$ satisfies a weak quantum stochastic differential equation of the following form: $\forall$ $u,v\in h,~f,g\in L^2(\IR_+,k_0)$ and $x\in\cla_0$,

\be   \label{fund}
\begin{split}
&\hskip-10pt\innerl j_t(x)ue(f),ve(g)\innerr\\
&=\innerl~xue(f),ve(g)\innerr+\sum_{\mu,\nu}\int_0^t ds \innerl~j_s(\theta^\mu_\nu(x))ue(f),ve(g)\innerr g^\mu(s)f_\nu(s),
\end{split}
\ee

or symbolically, 

\be\label{symbolicQSDE}
\begin{split}
dj_t(x)&=\sum_{\mu,\nu} j_s(\theta^\mu_\nu(x))\Lambda^\nu_\mu(ds);\\
&j_0(x)=x\ot I_\Gamma,
\end{split}
\ee

where we set $f_0(s)=f^0(s)=g_0(s)=g^0(s)=1$.

\end{enumerate}
 If $\cla$ is a von-Neumann algebra, we define a (von-Neumann algebraic) quantum stochastic flow on it in a similar way,  replacing the norm-density of $\cla_0$ by the ultra-weak density and also assuming the normality of 
  $j_t$ for each $t$. 
  
A quantum stochastic flow $(j_t)_{t\geq0}$ is called a $*$-homomorphic flow if furthermore 
$j_t$ is a $\ast$-homomorphism for each $t\geq0$.
\edfn
Often it is convenient to write the family of structure maps as a matrix $\Theta$ of maps given by:
\be\label{matrix}
\Theta:=
\left(
\begin{array}{ll}
\mathcal{L} & \delta^\dagger\\
\delta & \sigma\\
\end{array}
\right)
,\ee
where $\sigma:=\sum_{i,j}\theta^i_j(x)\ot|e_j><e_i|,$ $\delta(x):=\sum_{i}\theta^i_0(x)\ot |e_i\rangle$, $\delta^\dagger(x):=\sum_i\theta^0_i(x)\ot \langle e_i |$ and $\cll(x):=\theta^0_0(x),$ for $x\in\cla_0.$ We call $\Theta$ the structure matrix associated with the flow $(j_t)_{t\geq0}$.

\begin{rmrk}
Suppose that a family of contractive and completely positive flow $(j_t)_{t\geq0}$ satisfies (i) of Definition \ref{cocycle definition}. If furthermore $(j_t)_{t\geq0}$ satisfies the equation \eqref{fund} of Definition \ref{quantum stochastic flow definition} with bounded structure maps, then it follows that $(j_t)_{t\geq0}$ satisfies (ii) of Definition \ref{cocycle definition} (see page 171, Lemma 7.1.3 in \cite{dgkbs}).
\end{rmrk}

\begin{lmma}\label{homomorphism implies structure relations}
If $(j_t)_{t\geq0}$ is a $*$-homomorphic flow, then the structure maps $\{\theta^\mu_\nu\}_{\mu,\nu}$ satisfy :

\be\label{structure relations bdd}
\theta^\mu_\nu(xy)=\theta^\mu_\nu(x)y+x\theta^\mu_\nu(y)+\sum_{i \in \cli}\theta^i_\nu(x)\theta^\mu_i(y),~~ \theta^\mu_\nu(x)^*=\theta^\nu_\mu(x^*),
\ee
for $x\in\cla_0$.

\end{lmma}

\begin{proof}
The proof is an adaptation of the arguments given in Proposition 28.1 in page 234 of \cite{krp}. Since $t\mapsto j_t(x)$ is weakly measurable for all $x\in\cla$ and $(j_t)_{t\geq0}$ is a quantum stochastic cocyle, we have strong measurability of the map $t\mapsto j_t(x)ve(g)$ for all fixed $x\in\cla$, $v \in h$ and $g \in L^2(\IR_+,k_0)$. Hence we can re-phrase equation \eqref{fund} as
\be\label{strong integral}
\innerl ue(f),j_t(x)(ve(g))\innerr=\innerl u, xv \innerr e^{\innerl f,g\innerr}+\sum_{\mu,\nu}\int_0^tds~\innerl ue(f),j_s\left(\theta^\mu_\nu(x)\right)ve(g)\innerr f_{\mu}(s)g^{\nu}(s) .
\ee

Using the quantum It\^o formula as discussed in Subsection \ref{Fundamental martingales} we have:
\begin{equation*}
\begin{split}
0&=\innerl ue(f),\left(j_t(xy)-j_t(x)j_t(y)\right)ve(g)\innerr\\
&=\sum_{\mu,\nu}\int_0^tds~\innerl ue(f),j_s\left(\theta^\mu_\nu(xy)-\theta^\mu_\nu(x)y-x\theta^\mu_\nu(y)-\sum_{i \in \cli}\theta^i_\nu(x)\theta^\mu_i(y)\right)ve(g)\innerr f_\mu(s)g^\nu(s) 
\end{split}
\end{equation*}
from which it follows that $\theta^\mu_\nu(xy)=\theta^\mu_\nu(x)y+x\theta^\mu_\nu(y)+\sum_{i \in \cli}\theta^i_\nu(x)\theta^\mu_i(y)$.
Similarly, it follows from $j_t(x^*)=j_t(x)^*$ that $\theta^i_j(x^*)=(\theta^j_i(x))^*$ for all $i,j$.

\end{proof}

\begin{rmrk}\label{necessary for homo}
The relations \eqref{structure relations bdd} are equivalent to  the following identities:
\[
\pi(x):=\sigma(x)+x\ot I_{k_0}\quad(x\in\cla_0)
\]
is a $\ast$-homomorphism from $\cla_0$ to $\cla^{\prime\prime}\ot B(k_0)$.

\[
\delta(xy)=\delta(x)y+\pi(x)\delta(y)\quad(x,y\in\cla_0);
\]
\[
\cll(x^*)=(\cll(x))^*\quad(x\in\cla_0);
\]
\[
\delta(x)^*\delta(x)=\cll(x^*x)-\cll(x^*)x-x^*\cll(x)\quad(x\in\cla_0).
\]
\end{rmrk}
\begin{rmrk}
It follows from Lemma 7.1.3 of page 171 in \cite{dgkbs} that if the structure maps in Definition \ref{quantum stochastic flow definition} are norm bounded, then relations \eqref{structure relations bdd} are also sufficient for $(j_t)_{t\geq0}$ to be a $\ast$-homomorphic flow.
\end{rmrk}

\begin{lmma}\label{all semigroups are c0}
Let $(j_t)_{t\geq0}$ be a C* (respectively von-Neumann) algebraic quantum stochastic flow on $\cla$ with the noise space $k_0$, domain algebra $\cla_0$ and the structure matrix $\Theta=\begin{pmatrix}                                             
\cll&\delta^\dagger\\\delta&\sigma                                                                                                                 \end{pmatrix}$. Moreover in the von-Neumann algebraic case, we assume that $(j_t^{0,0})_{t\geq0}$ is $C_0$. Then $(j_t^{c,d}(\cdot))_{t\geq0}$ is a $C_0$-semigroup on $\cla$ with respect to the natural topology.  Furthermore,  the restriction of the generator of $(j_t^{c,d}(\cdot))_{t\geq0}$ to $\cla_0$ is $\cll+\innerl c,\delta \innerr+\delta^\dagger_d+\innerl c,\sigma_d\innerr+\innerl c,d\innerr id_\cla,$ where $\lgl c,\delta\rgl(x):=\lgl c,\delta(x) \rgl$, $\delta^\dagger_d(x):=\lgl d,\delta(x^*)\rgl^*$ and $\lgl c,\sigma_d\rgl(x):=\lgl c,\sigma(x)d\rgl$, for $x \in \cla_0$.  
\end{lmma}
\begin{proof}
It follows from Remark \ref{vacuum expectation semigroup} that $(j_t^{c,d})_{t\geq0}$ is a semigroup of maps on $\cla$. To prove the $C_0$ property, we proceed as follows. We shall most often write $c_t$ for $c1_{[0,t)},$
 where $c\in k_0,$  and obtain  for $x \in \cla_0$ 
\be\label{estimate for c0}
\begin{split}
&| \innerl u, \{ j_t^{c,d}(x)-x \}v \innerr|=|\innerl ue(c_t),j_t(x)ve(d_t)\innerr-\innerl u,xv\innerr|\\
&\leq |\innerl u(e(c_t)-e(0)),j_t(x)ve(0)\innerr|+|\innerl ue(0),j_t(x)ve(0)\innerr-\innerl u,xv\innerr|\\
&+|\innerl ue(c_t),j_t(x)v(e(d_t)-e(0)) \innerr|\\
&\leq \|u\|\|v\|\|x\|\left\{\sqrt{e^{t\|c\|^2}-1}+e^{\frac{t}{2}\|c\|^2}\sqrt{e^{t\|d\|^2}-1}\right\}+
|\innerl u,(j_t^{0,0}(x)-x)v \innerr|.
\end{split}
\ee

If $(j_t)_{t\geq0}$ is a C*-algebraic flow, it follows from equation \eqref{fund} that for $x\in\cla_0$, 
\[
\|j_t^{0,0}(x)-x\|\leq t\|\theta^0_0(x)\|.
\]
The norm density of $\cla_0$ in $\cla$, contractivity of the flow $(j_t)_{t\geq0}$ and the above estimate together imply that $(j_t^{0,0})_{t\geq0}$ is $C_0$ in the norm topology of $\cla$. Combining this with the estimate \eqref{estimate for c0}, it follows that $\lgl u,j_t^{c,d}(x)v\rgl\rightarrow \lgl u,xv\rgl$ as $t \rightarrow 0+$ uniformly for $u,v\in h$ such that $\|u\|\leq1$ and $\|v\|\leq1$. This implies that $(j_t^{c,d})_{t\geq0}$ is $C_0$ in the norm topology. 

On the other hand, if $(j_t)_{t\geq0}$ is a von-Neumann algebraic flow, then by hypothesis, the semigroup $(j_t^{0,0})_{t\geq0}$ is $C_0$ in the ultra-weak topology of $\cla$. Thus estimate \eqref{estimate for c0} implies that $j_t^{c,d}(x)\rightarrow x$ as $t\rightarrow 0+$ in the weak operator topology, hence also in the ultra-weak topology as $\| j_t^{c,d}(x)\|$ is uniformly bounded in $t$. This implies that $t\mapsto j_t^{c,d}(x)$ is ultra-weakly continuous for every $x\in\cla$.

Now for $x\in\cla_0$ and $u,v\in h,$ combining \eqref{fund} and \eqref{matrix}, we have
\be
\begin{split}
&\innerl ue(c1_{[0,t)}),j_t(x)ve(d1_{[0,t)})\innerr=\innerl u,xv\innerr e^{t\innerl c,d\innerr}\\
&+\innerl ue(c1_{[0,t)}),\{\int_0^t d\tau~ j_\tau(\cll(x)+\innerl c,\delta\innerr(x)+\delta^\dagger_d(x)+\innerl c,\sigma_d\innerr(x))\}ve(d1_{[0,t)})\innerr.\\
\end{split}
\ee
Thus 
\be
\begin{split}
&\text{lim}_{t\rightarrow 0+}\frac{1}{t}\lgl u,(j_t^{c,d}(x)-x)v\rgl=\lgl u,xv\rgl\lim_{t\rightarrow 0+}\frac{e^{t\lgl c,d\rgl}-1}{t}\\
&+\text{lim}_{t\rightarrow 0+}\lgl ue(c1_{[0,t)}),
\{\frac{1}{t}\int_0^t d\tau ~j_\tau(\cll(x)+\innerl c,\delta\innerr(x)+\delta^\dagger_d(x)+\innerl c,\sigma_d\innerr(x))\}ve(d1_{[0,t)})\rgl\\
&=\lgl u,\{\cll(x)+\lgl c,\delta\rgl(x)+\delta^\dagger_d(x)+\lgl c,\sigma_d\rgl(x)+\lgl c,d\rgl x\}v\rgl,\\
\end{split}
\ee
from which the conclusion follows.
\end{proof}

\bdfn 
 A $C_0$-semigroup of contractive (also normal in case of von-Neumann algebra) maps $(T_t)_{t\geq0}$ on $\cla$ is called a quantum dynamical semigroup if each $T_t$ is completely positive.
In case $\cla$ is unital, then a quantum dynamical semigroup $(T_t)_{t\geq0}$ is called  conservative   if 
$T_t(1)=1$ also holds for all $t$.
\edfn

\brmrk 
For a quantum stochastic flow satisfying the hypotheses of Lemma \ref{all semigroups are c0} the semigroup $(j_t^{0,0})_{t\geq0}$ is a quantum dynamical semigroup. It is called the vacuum expectation semigroup of the flow $j_t$.
\ermrk

\begin{rmrk}\label{wiener ito segal}
It is well known that if a classical stochastic process has chaotic decomposition property (\cite{meyer},\cite{krp}), then one can associate a quantum stochastic $\ast$-homomorphic flow with the process (see \cite{meyer},\cite{krp}). The class of stochastic processes having this property is rich and includes processes driven by Brownian motion and Poisson process.
\end{rmrk}

\subsection{Tensor product of Banach spaces}
Here we collect a few facts about the projective tensor product of Banach spaces which is an important technical tool, needed to prove our main result. Recall that (see \cite{takesaki}) for 
 two Banach spaces $E_1$ and $E_2$, the projective tensor product $E_1 \ot_\gamma E_2$ is the completion of the algebraic tensor product $E_1 \ot_{\rm alg} E_2$ in the cross-norm $\| \cdot \|_\gamma$
 given by $ \| X\|_\gamma={\rm inf} \sum_i \| x_i\| \|y_i\|$ for $X \in E_1\ot_{\rm alg} E_2$,  where the infimum is taken over all possible expressions of $X$ of the form $X=\sum_{i=1}^p x_i\ot y_i, ~x_i\in E_1,~ y_i\in E_2$ and $p\geq0$. \\


\begin{lmma}\label{bilinear}
Suppose that $T_j\in B(E_j,F_j)$  where $E_j,F_j$ for $j=1,2$ are Banach spaces. Then $T_1\ota T_2$ extends to a bounded operator  $$T_1\otp T_2:E_1\otp E_2\longrightarrow F_1\otp F_2$$ with the bound $$\|T_1\otp T_2\|\leq\|T_1\|\|T_2\|.$$
\end{lmma}

 \begin{proof} 
 The proof of Lemma \ref{bilinear} is an easy consequence of the estimate
\be
\begin{split}
\hspace{-300cm}&\|(T_1\ota T_2)(\sum_{i=1}^k x_i\ot y_i)\|_\gamma\leq\sum_{i=1}^k\|T_1(x_i)\ot T_2(y_i)\|_\gamma\\
&\leq \sum_{i=1}^k\|T_1(x_i)\|_{F_1}\|T_2(y_i)\|_{F_2}\leq \|T_1\|\|T_2\|\sum_{i=1}^k\|x_i\|_{E_1}\|y_i\|_{E_2}.\\
\end{split}
\ee
\end{proof}

\begin{lmma}\label{semigroup}
Suppose that $(T_t)_{t\geq0}$ and $(S_t)_{t\geq0}$ are two $C_0$-semigroups of bounded operators on Banach spaces $E_1$ and $E_2$, with generators $L_1$ and $L_2$ respectively. Then $(T_t\otp S_t)_{t\geq0}$ becomes a $C_0$-semigroup of operators on $E_1\otp E_2$ whose generator is the closed extension of the operator $L_1\ota I_{E_2}+I_{E_1}\ota L_2$ (defined on $D(L_1)\ota D(L_2)$), the closure being taken with respect to $\|\cdot\|_\gamma$.
\end{lmma}

\begin{proof}
We have $\left((T_t\ot_{\rm alg} S_t)\circ(T_s\ot_{\rm alg} S_s)\right)(X)=(T_{t+s}\ot_{\rm alg} S_{t+s})(X)$ for $X\in E_1\ota E_2$. Both sides being continuous in $\|\cdot\|_\gamma$, the above identity extends by Lemma \ref{bilinear} to $E_1\otp E_2$. Thus we have the semigroup property: 
$$(T_t\otp S_t)\circ(T_s\otp S_s)=(T_{t+s}\otp S_{t+s}).$$   By similar arguments $(T_t\otp I_{E_2})\circ(I_{E_1}\otp S_t)=T_t\otp S_t$ and thus the strong continuity of $T_t\otp I_{E_2}$ and $I_{E_1}\otp S_t$, as a function of t yields the strong continuity of $T_t\otp S_t.$  Hence $(T_t\otp S_t)_{t\geq0}$ is a  $C_0$-semigroup on $E_1\otp E_2$.  Moreover $T_t\otp S_t$ keeps $D(L_1)\ota D(L_2)$ invariant. Thus $D(L_1)\ota D(L_2)$ is a core for the generator of $T_t\otp S_t$ (see \cite{davies}). We will denote the generator by $L_1\otp I_{E_2}+I_{E_1}\otp L_2$. Clearly, this is the closure of the operator $L_1\ota I_{E_2}+I_{E_1}\ota L_2$.
\end{proof}

\begin{crlre}\label{projective core}
In the notation of Lemma \ref{semigroup}, suppose that $\cld_1\subset E_1$ and $\cld_2\subset E_2$ are cores for the generators $L_1$ and $L_2$ respectively. Then $\cld_1\ot_{\rm alg}\cld_2$ is a core for $L_1\otp I_{E_2}+I_{E_1}\otp L_2$. 
\end{crlre}

\begin{proof}
 Recall that $D(L_1)\ot_{\rm alg} D(L_2)$ is a core for the generator $L_1\otp I_{E_2}+I_{E_1}\otp L_2$. Let $X:=\sum_{i=1}^k x_i\ot y_i\in D(L_1)\ot_{\rm alg} D(L_2)$, where $x_i\in D(L_1)$ and $y_i\in D(L_2)$, $i=1,\ldots,k$. Get $(x^{(i)}_n)_n\in \cld_1$ and $(y^{(i)}_n)_n\in\cld_2$, $i=1,\ldots,k$ such that $x^{(i)}_n\rightarrow x_i$, $L_1(x^{(i)}_n)\rightarrow L_1(x_i)$, $y^{(i)}_n\rightarrow y_i$, $L_2(y^{(i)}_n)\rightarrow L_2(y_i)$. Let 
$X_n:=\sum_{i=1}^k x^{(i)}_n\ot y^{(i)}_n$. Then it follows that $X_n\rightarrow X$ and $(L_1\otp I_{E_2}+I_{E_1}\otp L_2)(X_n)\rightarrow (L_1\otp I_{E_2}+I_{E_1}\otp L_2)(X)$ in $\|\cdot\|_\gamma$. This proves the result.
\end{proof}

We conclude this section with the following result about operators on Banach spaces, which will play a crucial role in the proof of homomorphism property in Section 3.
\begin{lmma}\label{dg}
Let $E$ be a Banach space, and let $A$ and $T$ belong to $Lin(E,E)$ with dense domains $D(A)$ and $D(T)$ respectively. Suppose that there is a total set $D\subset D(A)\cap D(T)$ with the properties :
\begin{enumerate}
\item[] {\bf(i)} $A(D)$ is total in $E$,  ~~{\bf(ii)} $\|T(x)\|<\|A(x)\|$ for all $x\in D.$ 
\end{enumerate}
Then $(A+T)(D)$ is also total in $E$.
\end{lmma}
\begin{proof}
Set $F:= span\{(A+T)(D)\}$ and suppose that $\overline{F}\neq E$. Then noting that if $A(D)\subseteq(A+T)(D)$, one has  $\overline{F}\supset\text{span}~\overline{A(D)}=E$ (by hypothesis), we conclude that there exists a  $y_0(\neq0)$ in $A(D)$, such that $y_0\notin\overline{F}$. Let $y_0=A(x_0)$ for some $x_0\in D.$ Then by Hahn-Banach theorem, there exists $\Lambda\in E^\ast,$ the topological dual of $E$, such that $\|\Lambda\|=1,$  $|\Lambda(y_0)|=\|y_0\|$ as well as $\Lambda((A+T)(D))=0.$ Then $\|y_0\|=|\Lambda(A(x_0))|$ and $|\Lambda(A(x_0))|=|\Lambda(T(x_0))|.$\\ But $|\Lambda(T(x_0))|\leq\|T(x_0)\|<\|A(x_0)\|=\|y_0\|$ which leads to a contradiction. Therefore $\overline{F}=E.$
\end{proof}

\section{A sufficient condition for a quantum stochastic flow to be $\ast$-homomorphic}

\subsection{Assumptions and statement of the main result}\label{assumptions n impli}

Let $(j_t)_{t\geq0}$ be a quantum stochastic flow with the noise space $k_0$, domain algebra $\cla_0$ and structure matrix $\Theta=\begin{pmatrix}
\cll&\delta^\dagger\\\delta&\sigma                                                                                                                                                                                                                                                                                   \end{pmatrix}
$. Our aim is to give a set of natural and sufficient conditions which will imply that the flow is $\ast$-homomorphic. In the following, note that the first assumption {\bf A(1)} is nothing but those given in Remark \ref{necessary for homo} which are known to be necessary by Lemma \ref{homomorphism implies structure relations}.

We now state the assumptions:

\begin{enumerate}

\item[{\bf A(1)}]
The map $\pi:\cla_0\rightarrow\cla^{\prime\prime}\ot B(k_0)$ given by:
\[
\pi(x):=\sigma(x)+x\ot I_{k_0}\quad(x\in\cla_0)
\]
is a $\ast$-homomorphism from $\cla_0$ to $\cla^{\prime\prime}\ot B(k_0)$.
\[
\cll(x^*)=(\cll(x))^*\quad(x\in\cla_0);
\]

\[
\delta(xy)=\delta(x)y+\pi(x)\delta(y)\quad(x,y\in\cla_0);
\]

\[
\delta(x)^*\delta(x)=\cll(x^*x)-\cll(x^*)x-x^*\cll(x)\quad(x\in\cla_0).
\]

\item[{\bf A(2)}]
There exists a semifinite, faithful and lower-semicontinuous trace $\tau$ on $\cla$, such that $\cla_0\subset D(\tau)$. Setting $h:=L^2(\tau)$, the G.N.S. Hilbert space of $\cla$ with respect to $\tau$ and viewing $\cla$ as a C*-subalgebra of $B(h)$, $\cla_0$ is furthermore assumed to be dense in $h$.

\item[{\bf A(3)}]
\begin{enumerate}
\item[(i)]
For each $t\geq0$, the semigroup of maps $(T_t)_{t\geq0}$ extends to a $C_0$ contractive semigroup $(T_t^{(2)})_{t\geq0}$ on $h$. Its generator is denoted  by $\cll_{2}$.

Furthermore, $\cla_0\subseteq D(\cll)\cap D(\cll_2)$ and is a core for $\cll_2$.

\item[(ii)]
For $x\in\cla_0$, $\cll(x^*x)\in\cla\cap L^1(\tau)$ and $\tau(\cll(x^*x))\leq0$ (a kind of weak dissipativity).

\item[(iii)]
The semigroup $(T_t^{(2)})_{t\geq0}$ is analytic.
\end{enumerate}

\item[{\bf A(4)}]
\[
\text{sup}_{0\leq s\leq t}|\lgl uf^{\ot^m} j_s(x)vg^{\ot^n}\rgl|\leq C(m,n,u,v,t,f,g)\|x\|_1\quad(u,v\in h,x\in\cla\cap L^1(\tau)),
\]
where $m,n\in\IN$, $C(m,n,u,v,t,f,g)=O(e^{\beta t})$ for $\beta>0$ and $f,g\in\clw$ for some total subset $\clw$ of $L^2(\IR_+,k_0)$. 
\end{enumerate}

We now state the main theorem of this paper:
\bthm\label{main theorem}
Let $(j_t)_{t\geq0}$ be a quantum stochastic flow with noise space $k_0$ and structure matrix $\Theta=\begin{pmatrix}\cll&\delta^\dagger\\\delta&\sigma\end{pmatrix}$. Furthermore, suppose that the flow satisfies the assumptions {\bf A(1)}--{\bf A(4)}. Then $j_t$ is a $\ast$-homomorphism for each $t\geq0$.
\ethm

\subsection{Remarks on the assumptions}

We begin with conditions in terms of the semigroups $(j_t^{c,d})_{t\geq0}$ associated with a quantum stochastic cocycle $(j_t)_{t\geq0}$, which will imply the condition given in assumption {\bf A(4)}.

\begin{lmma}\label{jevaabehoy}
Let $\clw$ be a total subset of $k_0$ such that $z \clw \subseteq \clw$ for all $z \in \IC$. Suppose that a quantum stochastic cocycle $(j_t)_{t\geq0}$ satisfies : 
\begin{enumerate}
\item[{\bf A(4)$^\prime$}]
\be\label{Qi}
\|j_t^{c,d}(x)\|_1\leq exp(tM)\|x\|_1
\ee 
\end{enumerate}
 for $ x\in\cla\cap L^1(\tau),$ $c,d \in \clw$ and  where $M\geq0$ depends only on $\|c\|,\|d\|.$ Then the estimate in {\bf A(4)} holds.
\end{lmma}
\begin{proof}
For a partition $0=s_0<s_1<s_2< \ldots <s_n=t$ and for functions of the form $f=\sum_{j}1_{[s_{j-1},s_j)}c_j,$~ $g=\sum_{j}1_{[s_{j-1},s_j)}d_j,$ ($c_j,d_j\in \clw~for~j=1,2,\ldots$), the cocycle property of $j_t(\cdot)$ and (\ref{Qi}) together imply:  $$\|\innerl e(f),j_t(x) e(g) \innerr\|_1\leq exp(tM)\|x\|_1,$$ where $M=max_j(M_j),$ and each $M_j$ depends only on $\|c_j\|$ and $\|d_j\|.$
Let $\Lambda(z):=\innerl ue(\bar{z}f),j_t(x)ve(g)\innerr$. We have $|\innerl ue(\bar{z}f),j_t(x)ve(g)\innerr|\leq exp(tM)\|x\|_1$ for $|z|=1.$ Clearly $\Lambda$ is entire in z since $z\mapsto e(zf)$ is strongly entire and by considering a unit disc centered at zero and applying Cauchy's estimate to this function, we obtain:
\be
(m!)^{\frac{1}{2}}|\innerl uf^{\otimes^m},j_t(x)ve(g)\innerr|\leq \|u^*v\|_\infty m!exp(tM)\|x\|_1,
\ee
for $u,v\in\cla\cap L^2(\tau)$ and $x\in\cla\cap L^1(\tau).$
A similar calculation with the function $\beta(z):=\innerl uf^{\otimes^m},j_t(x)ve(zg)\innerr$ yields:
\be\label{cauchy}
\begin{split}
&|\innerl uf^{\otimes^m},j_t(x)vg^{\otimes^n}\innerr|\leq \|u^*v\|_\infty (m!n!)^{\frac{1}{2}}exp(tM))\|x\|_1,
\end{split}
\ee
which proves that the cocycle $(j_t)_{t\geq0}$ satisfies {\bf A(4)}, if we take\\ $C(m,n,u,v,t,f,g):=\|u^*v\|_\infty (m!n!)^{\frac{1}{2}}exp(tM).$ 
\end{proof}

For certain special types of von-Neumann algebras, the estimate in {\bf A(4)} is automatic, as the next result shows. Note that a type-I von-Neumann algebra $\cla$ with atomic center and acting on a separable Hilbert space, is a direct sum  of the form $\cla \cong  \bigoplus_{m \in T} \cla_m \ot l^\infty(\Gamma_m)$ where $T$ is some index set such that for each $m\in T$, $\cla_m$ is isomorohic to $\clb(H)$ (where $H$ is a separable Hilbert space of finite or infinite dimension) and $\Gamma_m$ is a discrete, at most countable set. There is a natural faithful and semifinite trace $\tau$ on $\cla$, which is a direct sum (over $T$) of the tensor product of the canonical semifinite trace on $\clb(H)$ and the trace on $l^\infty(\Gamma_m)$ coming from the counting measure on $\Gamma_m$. It is clear that $\| x \|_\infty \leq \| x \|_1$ for all $x \in \cla \bigcap L^1(\tau)$.

\begin{lmma}\label{qs flow is homomorphic for type-1}
Let $\cla$ be a type-I von-Neumann algebra with atomic center, acting on a seperable Hilbert space. Suppose that $\tau$ is the natural trace on $\cla$ as discussed above. Then any quantum stochastic flow $(j_t)_{t\geq0}$ on $\cla$ satisfying the assumtions {\bf A(1)}--{\bf A(3)} must also satisfy the estimate in {\bf A(4)}.
\end{lmma}

\begin{proof}
Using $\|x\|_\infty\leq\|x\|_1$ and the contractivity of  $j_t$, we have  for $x\in L^1(\tau)$

\be
\begin{split}
&\sup_{0\leq s\leq t}|\innerl uf^{\ot^m},j_t(x)vg^{\ot^n}\innerr|\\
&\leq\|x\|_\infty\|f^{\ot^m}\|\|g^{\ot^n}\|\|u\|_2\|v\|_2\\
&\leq\|x\|_1\|f^{\ot^m}\|\|g^{\ot^n}\|\|u\|_2\|v\|_2,
\end{split}
\ee which implies that the flow $(j_t)_{t\geq0}$ satisfies {\bf A(4)} with $C(m,n,u,v,t,f,g):=\|f^{\ot^m}\|\|g^{\ot^n}\|\|u\|_2\|v\|_2$.
\end{proof}

Let us now discuss a few special cases to show that the assumptions made above are general enough to cover many interesting and natural examples arising in classical probability.

We state without proof the following proposition (see pp.63, Lemma 3.2.28 of \cite{dgkbs}).

\begin{ppsn}\label{symmetric qds}
Let $\cla$ be a unital C*-algebra with a faithful, semifinite and lower-semicontinuous trace $\tau$ and $(j_t)_{t\geq0}$ be a quantum stochastic flow with the noise space $k_0$ and structure matrix $\Theta=\begin{pmatrix}\cll&\delta^\dagger\\\delta&\sigma\end{pmatrix}$. Suppose furthermore that with respect to the trace $\tau$:
\[
\hspace{2.0cm}T_t(1)=1,~\tau(T_t(x)y)=\tau(xT_t(y))\quad(t\geq0;x,y\in D(\tau)),
\]
where $(T_t)_{t\geq0}$ is the vacuum expectation semigroup associated with the flow $(j_t)_{t\geq0}$. Also suppose that the assumptions {\bf A(1)} and {\bf A(2)} hold. Then {\bf A(3)} follows.
\end{ppsn}


The following result shows the abundance of classical as well as noncommutative quantum stochastic flows for which the estimate ${\bf A(4)^\prime}$ holds. 

\begin{lmma}\label{asol1}
Let $\cla$ be a C*-algebra equipped with a faithful, semifinite, lower semicontinuous trace $\tau$ and assume  that there is a strongly continuous, $\ast$-automorphic action $\alpha_g$ of  a  Lie group $G$  on $\cla$ which is $\tau$-preserving, that is $\tau(\alpha_g(a))=\tau(a)$.  Extend $\alpha_g$ to a unitary operator $U_g$ on $L^2(\tau)$ and we extend $\alpha$ to $\cla^{\prime\prime}$ as a normal $\ast$-automorphism given by $\alpha_g(x)=U_gxU_g^*.$
Let $(g_t)_{t\geq0}$ be a $G$-valued L\'evy process, defined on some probability space $(\Omega,\clf,P).$ Define $j_t:\cla^{\prime\prime}\rightarrow L^\infty(\Omega,\cla^{\prime\prime})~\subseteq B(L^2(\tau)\ot L^2(\Omega)),$ by $j_t(x)(\omega):=\alpha_{g_t(\omega)}(x)$ and let  $T_t(x)=\IE(\alpha_{g_t(\omega)}(x)).$ 
\begin{enumerate}
\item[(i)] Then for $f,g\in L^2(\Omega),~x\in\cla\cap L^1(\tau),$ we have 
\be
\tau\left(\left|\innerl f,j_t(x)g \innerr\right|\right) 
\leq\|f\|_2\|g\|_2\|x\|_1.
\ee

\item[(ii)] $(T_t)_{t\geq0}$ is a normal quantum dynamical semigroup on $\cla^{\prime\prime}$ and if its restriction on $\cla$ leaves $\cla$ invariant, it is a quantum dynamical semigroup on $\cla$. Furthermore if $g_t$ and $(g_t)^{-1}$ have the same distribution for each $t$, then the semigroup is $\tau$-symmetric.
\end{enumerate} 
\end{lmma}

\begin{proof}
To prove (i), it suffices to show the inequality for positive $x\in\cla.$  For such $x,$ we have 
\be
\begin{split}
&\tau\left(\left|\innerl f,j_t(x)g \innerr\right|\right)\\
&\leq\int_\Omega dP\tau\left(\left|\overline{f}(\omega)g(\omega)j_t(x)(\omega)\right|\right)\\
\end{split}
\ee
from which the result follows.

To prove (ii) we proceed as follows:

from the defining property of L\'evy processes, the semigroup property of $T_t$ follows; while the normality of $T_t$ is a consequence of the fact that $j_t$ is implemented by a normal automorphism of $\cla^{\prime\prime}.$ Moreover, if $g_t$ and $g_t^{-1}$ have the same distribution, we have for $a,b \in L^1(\tau) \bigcap \cla$
\be
\begin{split}
&\tau(T_t(a)b)=\tau[\IE\{\alpha_{g_t(\omega)}(a)b\}]\\
&=\tau[\IE\{\alpha_{g_t(\omega)}(a\alpha_{g_t(\omega)^{-1}}(b))\}]=\IE[\tau\{\alpha_{g_t(\omega)}(a\alpha_{g_t(\omega)^{-1}}(b)) \}]\\
&=\IE[\tau\{(a\alpha_{g_t(\omega)^{-1}}(b))\}]=\tau[a\IE\{\alpha_{g_t(\omega)}(b)\}]=\tau(aT_t(b)).\\
\end{split}
\ee
\end{proof}

Note that if the L\'evy process in Lemma \ref{asol1} has the chaotic decomposition property, for example if it is a Brownian motion or a 	Poisson process, then we can realize the flow $j_t$ as a quantum stochastic flow in a suitable Fock space. To see a concrete example, assume furthermore that $G$ be a second countable, compact Lie group of dimension $k$, acting smoothly on  $\cla$.  Denote by $\cla^\infty$ the dense $\ast$-subalgebra of $\cla$ generated by elements $x$ such that $g\rightarrow\alpha_g(x)$ is norm-smooth, where $g\rightarrow\alpha_g$ is the group action. Let $\{\chi_\ell\}_{\ell=1}^k$ be a basis for the Lie algebra of $G$ and let $g_t$ be the $G$-valued standard Brownian motion. Then $g_t$ and $g_t^{-1}$ have the same distribution, so the vacuum expectation semigroup of $j_t$ is symmetric. Moreover taking $\cla_0=\cla^\infty$,  the assumptions {\bf A(1)}, {\bf A(2)}, {\bf A(3)}  are easily seen to hold. By Lemma \ref{asol1} ${\bf A(4)^\prime}$  and  hence {\bf A(4)} holds too.  
The flow $j_t$ can be viewed as a possibly noncommutative Brownian motion on $\cla$.

\begin{rmrk}
Consider a typical diffusion process in $\IR$ whose generator is of the form:
$$\cll=\frac{1}{2}\frac{d}{dx}a^2(x)\frac{d}{dx}+b(x)\frac{d}{dx}.$$
The coefficients $a$ and $b$ are assumed to be smooth and $a$ is assumed to be non-vanishing everywhere. By a change of variable $x\mapsto \phi(x)$, where $\phi(x)=\int_0^x ds e^{\int_0^s\ dt \frac{2b(t)}{a^2(t)}}+C$ (where C is a constant), the generator $\cll$ can be made symmetric with respect to the trace $\tau^\prime$ given by $\tau^\prime(f)=\int f(x)\phi^\prime(x)dx.$ Thus the assumption of symmetry can accommodate the semigroup corresponding to an arbitrary one-dimensional diffusion with smooth coefficients. 
\end{rmrk}

\begin{rmrk}
On the other hand, for some of the most common Markov processes arising in classical probability for which the corresponding quantum stochastic flow (see Remark \ref{wiener ito segal}) satisfies assumptions {\bf A(1)}--{\bf A(3)}, the vacuum expectation semigroup $(T_t)_{t\geq0}$ associated with the flow cannot be made symmetric even by a change of measure on the underlying function algebra. For example, consider the flow associated to the standard Poisson process on $\IZ_+,$ realized on the commutative von-Neumann algebra $l^\infty(\IZ_+),$ equipped with the trace given by the counting measure. Here the generator $\cll_2$ is the bounded operator $l-I,$ where l denotes the unilateral shift operator on $l^2(\IZ_+).$ It is straightforward to see that for $\phi~\in~l^\infty(\IZ_+),~\phi\geq0,$ $\tau((l-I)(\phi))\leq0.$ However, there does not exist any faithful positive trace for which $\cll$ is symmetric, as seen from the following argument:

If $\cll$ is symmetric with respect to some measure $\mu,$ say given by a sequence $\{p_i=\mu(\{i\})\}$ of non-negative numbers, then the symmetry condition will imply (since $1\in D(\cll)$) that $\sum_{i\geq0}(\phi(i+1)-\phi(i))p_i=0$ for all $\phi\in l^\infty(\IZ_+)$. Hence we have $p_i=0$ for all i. 
\end{rmrk}

\subsection{Proof of the main theorem}

Throughout this subsection let us fix a quantum stochastic flow $(j_t)_{t\geq0}$ with the noise space $k_0$ and structure matrix $\Theta=\begin{pmatrix}\cll&\delta^\dagger\\\delta&\sigma\end{pmatrix}$, satisfying the assumptions {\bf A(1)}--{\bf A(4)}.

We need a few preparatory lemmas to prove Theorem \ref{main theorem}. We begin with the following observations:

\begin{rmrk}\label{interchange}
Let $E$ be a Banach space, $F:\IR\rightarrow E$ be a strongly measurable map and let $\mu$ be a measure on $\IR$. Suppose that the integrals $\int_{\IR} d\mu(t) F(t)$ and $\int_{\IR} d\mu(t) T(F(t))$ exist, where $T$ is a closed densely defined operator in $E.$ Then $\int_{\IR} d\mu(t) F(t)\in D(T)$ and $$T(\int_{\IR} d\mu(t) F(t))=\int_{\IR} d\mu(t) T(F(t)).$$
\end{rmrk}

\begin{rmrk}\label{analytic semigroup n real part}
We observe that because of analyticity in assumption {\bf A(3)}, the real part of the operator $(-\cll_2)$ exists as an unbounded, densely defined and non-negative operator (see pages 322 and 336 of \cite{kato}). 
\end{rmrk}

Throughout this subsection, we will be working in the projective tensor product $h\otp h$. We fix the following notations:
\[
 \hat{\cll}:=\cll_2\otp I+I\otp\cll_2;
\] 

\[
L:=(-2Re(\cll_2))^{\frac{1}{2}};
\] 

\[
L\otp L:=(L\otp I)\circ(I\otp L)=(I\otp L)\circ(L\otp I);
\] 

\[
\clf:=\cla_0\ota\cla_0;
\] 

\[
\cly:=\{{(\lambda-\hat{\cll})^{-1}(x\ot y)|~x,y\in\cla_0}\}\quad(\lambda>0).
\]

The next two lemmas set the stage for the application of Lemma \ref{dg} to our problem, leading to the proof of Theorem \ref{main theorem}.

\begin{lmma} \label{ineq1}
For $x\in\cla_0$ with $x\neq0$ and $\lambda>0$, we have the following strict inequality: $$\int_0^\infty e^{-\lambda t} dt \|L(T_t(x))\|^2<\|x\|^2.$$ 
\end{lmma}

\begin{proof} For x in $\cla_0$,
\be
\frac{d}{dt}\|T_t(x)\|^2=\innerl \cll_2(T_t(x)),T_t(x)\innerr+\innerl T_t(x),\cll_2(T_t(x))\innerr=-\|L(T_t(x))\|^2.
\ee
For $\lambda>0$, we have
\be
\begin{split}
\int_0^\infty e^{-\lambda t} \|L(T_t(x))\|^2dt=&-\int_0^\infty e^{-\lambda t} \frac{d}{dt}\|T_t(x)\|^2dt\\=&\left\{\|x\|^2-\lambda\int_0^\infty e^{-\lambda t} \|T_t(x)\|^2dt\right\},
\end{split}
\ee
which implies that 
\[
\int_0^\infty e^{-\lambda t} \|L(T_t(x))\|^2 dt \leq\|x\|^2\quad(x\in\cla_0).
\]

If $\int_0^\infty dt~ e^{-\lambda t} \|T_t(x)\|^2=0$ for some $\lambda>0$, we have $\|T_t(x)\|=0$ for all t. By the continuity of $\|T_t(x)\|$ as a function of t, this implies that $x=0$, which leads to a contradiction. 
\end{proof}

\begin{lmma}\label{extension of C}
$\|(L\otp L)(X)\|_\gamma\leq\|(\lambda-\hat{\cll})(X)\|_\gamma$ for all $X$ in $D(\hat{\cll})$ and $\lambda>0$. Furthermore we have strict inequality if $X$ is in $\cly$.
\end{lmma}

\begin{proof}




Let $X\in\clf$ such that $X=\sum_{i=1}^k x_i\ot y_i$. It is obvious that\\ $( I\otp L)(\clf)\subset D((L\otp I))$. So using Lemma \ref{ineq1} we have:
\be \label{equa10}
\begin{split}
\hskip-10pt&\int_0^\infty dt ~e^{-\lambda t}\|\{(L\otp L)(T_t\otp T_t)\}(X)\|_\gamma\\
&=\int_0^\infty dt~e^{-\lambda t}\|\sum_{i=1}^k L(T_t(x_i))\ot L(T_t(y_i))\|_\gamma\\
&\leq \sum_{i=1}^k(\int_0^\infty dt~e^{-\lambda t}\|L(T_t(x_i))\|^2)^{\frac{1}{2}}(\int_0^\infty dt~e^{-\lambda t}\|L(T_t(y_i))\|^2)^{\frac{1}{2}}	
<\sum_{i=1}^k\|x_i\|\|y_i\|.
\end{split}
\ee
Equation (\ref{equa10}) and Remark \ref{interchange} together yield:
$$\|(L\otp L)((\lambda-\hat{\cll})^{-1}(X))\|_\gamma\leq\|X\|_\gamma.$$
Thus $(L\otp L)\circ(\lambda-\hat{\cll})^{-1}$ is a contractive operator in $h\otp h$. As a consequence, $L\otp L$ extends to $D(\hat{\cll})$ by (i) of {\bf A(3)} and Corollary \ref{projective core}. This gives us the required inequality. 

Let $X=x\otimes y,$ where $x,y\in\cla_0$. The above equations give: 
\be
\begin{split}
&\|(L\otp L)((\lambda-\hat{\cll})^{-1}(x\ot y))\|_\gamma <\|x\|\|y\|=\|x\ot y\|_\gamma\\
&\hskip-15pt\mbox{or}
~~\|(L\otp L)(Y)\|_\gamma <\|(\lambda-\hat{\cll})(Y)\|_\gamma
\hskip10pt\mbox{for $Y\in\cly$,}
\end{split}
\ee
which proves the second claim.
\end{proof}

\begin{lmma}\label{right kind of delta}
For $x\in\cla_0$ we have:
\begin{enumerate}

\item[(i)]
\[
\theta^i_0(x)\in h(=L^2(\tau))\quad(i\geq1).
\]

\item[(ii)]
The series $\sum_{i\geq1}\theta^i_0(x)\ot |e_i\rgl$ is convergent in the norm of $h\ot k_0$. Thus we have:
\[
\delta(x)=\sum_{i\geq1}\theta^i_0(x)\ot |e_i\rgl\in h\ot k_0.
\]

\item[(iii)]
There exists a densely defined positive operator $C$ such that $D(\cll_2)$ is a core for $C$ and 
\[
Cu=Lu\quad(u\in D(\cll_2)); 
\]

\[
Re\lgl u, (-\cll_2)v\rgl=\lgl Cu,Cv\rgl\quad(u,v\in D(\cll_2)). 
\]

For each $i\geq1$, $\theta^i_0$ extends to a linear map from $D(C)$ to $h$.

\item[(iv)]
There exists a linear map $B:D(\hat{\cll})(\subseteq h\otp h)\rightarrow h\otp h$, satisfying
\begin{enumerate}
\item[(a)]
\[
\|B(X)\|_\gamma\leq\|(\lambda-\hat{\cll})(X)\|_\gamma\quad(X\in D(\hat{\cll})),
\]
the inequality being strict if $X\in\cly$.
\item[(b)]
\[
B(x\ot y)=\sum_{i\geq1}\theta^i_0(x)\ot \theta^i_0(y)\quad(x,y\in\cla_0).
\]
\end{enumerate}

\end{enumerate}

\end{lmma}

\begin{proof}
For $x\in\cla_0$, we have 
\[
\tau(\theta^i_0(x)^*\theta^i_0(x))\leq\tau(\sum_{i\geq1}\theta^i_0(x)^*\theta^i_0(x))=\tau(\delta(x)^*\delta(x)).
\]
The last identity in {\bf A(1)} implies that:
\[
\tau(\delta(x)^*\delta(x))=\tau(\cll(x^*x))-\tau(\cll(x^*)x)-\tau(x^*\cll(x))\quad(x\in\cla_0).
\]
Combining this with (ii) of {\bf A(3)} we have:
\[
\tau(\delta(x)^*\delta(x))\leq -\tau(\cll(x^*)x)-\tau(x^*\cll(x))<\infty\quad(x\in\cla_0),
\]
which proves (i) and (ii). 

To prove (iii) note that the last identity in {\bf A(2)} implies the following:
\be\label{asolmaps}\begin{split} 
\|\theta^i_0(x)\|_2^2\leq&\sum_{j=1}^\infty\|\theta^j_0(x)\|^2_2\leq\tau(\delta(x)^*\delta(x))\\
&\leq\tau(\left(-\cll\right)(x^*)x)+\tau(x^*\left(-\cll\right)(x))\leq\|L(x)\|^2_2\leq\|(L+\epsilon)(x)\|^2_2,
\end{split}
\ee 
for all x in $\cla_0,~\epsilon>0$ and $i\geq1$, as $L$ is a non-negative operator. By Remark \ref{analytic semigroup n real part}, we have
\[
Re\lgl u,(-\cll_2)v\rgl=\lgl u, L^2v\rgl= \lgl Lu,Lv\rgl\quad(u,v\in D(\cll_2)).
\]
Thus by Theorem 1.27 of page 318 in \cite{kato}, the densely defined, non-negative sesquilinear form $q(u,v):=\lgl u,L^2v\rgl$, $u,v\in D(\cll_2)$ is closable. Let $T$ be the positive operator associated with the form $q$, as obtained in Theorem 2.23 of page 331 in \cite{kato}. Set $C:=T^{\frac{1}{2}}$. Note that $Tu=L^2u$ for all $u\in D(\cll_2)$ and $D(\cll_2)$ is a core for $C$. Thus from equation \eqref{asolmaps}, we have

\be\label{asolmap}\begin{split} 
\|\theta^i_0(x)\|_2^2\leq&\sum_{i=1}^\infty\|\theta^i_0(x)\|^2_2\leq\tau(\delta(x)^*\delta(x))\\
&\leq\tau(\left(-\cll\right)(x^*)x)+\tau(x^*\left(-\cll\right)(x))\leq\|C(x)\|^2_2\leq\|(C+\epsilon)(x)\|^2_2,
\end{split}
\ee 
for $x\in\cla_0$ and $i\geq1$, from which (iii) follows.

Finally we prove (iv):

Set $C\otp C:=(I\otp C)\circ(C\otp I)=(C\otp I)\circ(I\otp C)$. Define a map $B$ belonging to $Lin(D(C)\ota D(C),h\otp h)$ by: $$B(x\ot y)=\sum_{i\geq1}\theta^i_0(x)\ot\theta^i_0(y)\quad(x,y\in D(C),i\geq1),$$ and extending linearly. This operator is well-defined because 
for $\epsilon>0$, 
\be
\begin{split}
&\sum_{i\geq1}\|\theta^i_0(x)\|\|\theta^i_0(y)\|\leq\{(\sum_{i\geq1}\|\theta^i_0(x)\|^2)(\sum_{i\geq1}\|\theta^i_0(y)\|^2)\}^\frac{1}{2}\\
&\leq\|(C+\epsilon)x\|\|(C+\epsilon)y\|
<\infty.
\end{split}
\ee

It follows that  

\be 
\|B\{(C+\epsilon)^{-1}\otp(C+\epsilon)^{-1}\}(x\ot y)\|_\gamma\leq\|x\ot y\|_\gamma,
\ee

which implies that $B\circ\{(C+\epsilon)^{-1}\otp(C+\epsilon)^{-1}\}$ extends to a contraction on $h\otp h$. Thus we have
\be
\|B(X)\|_\gamma\leq\|\{(C+\epsilon)\otp(C+\epsilon)\}(X)\|_\gamma
\ee
for all $X\in D(C)\ota D(C)$ and $\epsilon>0$. Letting $\epsilon\rightarrow0$ we get 
\be
\|B(X)\|_\gamma\leq\|(C\otp C)(X)\|_\gamma
\ee
for all X in $D(C)\ota D(C).$ Note that the conclusion of Lemma \ref{extension of C} holds with $L$ replaced by $C$. As $C\otp C$ extends to $D(\hat{\cll})$, we can also extend $B$ to $D(\hat{\cll}).$ So we have:
\be\label{bittu}
\|B(X)\|\leq\|(C\otp C)(X)\|_\gamma\leq\|(\lambda-\hat{\cll})(X)\|_\gamma~\mbox{for all $X\in D(\hat{\cll})$.}
\ee
Since $\cly\subseteq D(\hat{\cll})$, we have:
\be\label{B ineq}
\|B(Y)\|_\gamma\leq\|(C\otp C)(Y)\|_\gamma<\|(\lambda-\hat{\cll})(Y)\|_\gamma\quad(Y\in\cly).
\ee
This proves (iv).
\end{proof}

We are now in a position to prove Theorem \ref{main theorem}.\\

{\bf Proof of Theorem \ref{main theorem}:}

Throughout the proof we adopt Einstein's summation convention for convenience. 

For $f,g$ in $\clw,$ the flow equation \eqref{fund} leads to : 
\be
\innerl j_t(x)ue(f),ve(g)\innerr=\innerl xue(f),ve(g)\innerr+\int_0^t ds \innerl j_s(\theta^\mu_\nu(x))ue(f),ve(g)\innerr g^\mu(s)f_\nu(s).
\ee
Using the quantum It\^o formula we get:
\be   \label{ito}
\begin{split}
&\hskip-15pt\innerl j_t(x)ue(f),j_t(y)ve(g)\innerr\\=&\innerl xue(f),yve(g)\innerr
+\int_0^t ds[\innerl j_s(\theta^\mu_\nu(x))ue(f),j_s(y)ve(g)\innerr g^\mu(s)f_\nu(s)\\
&+\innerl j_s(x)ue(f),j_s(\theta^\mu_\nu(y))ve(g)\innerr f^\mu(s)g_\nu(s)
\\&+\innerl~j_s(\theta^i_\mu(x))ue(f),j_s(\theta^i_\nu(y))ve(g)\innerr~f_\mu(s)g_\nu(s)].
\end{split}
\ee

For fixed $u,v$ in $\cla\cap h,$ $f,g$ in $\clw,$ we define for each $t\geq0,$\\ $\phi_t:\cla_0\times\cla_0\rightarrow\IC$ by
\be
\phi_t(x,y):=\innerl~j_t(x)ue(f),j_t(y)ve(g)\innerr~-\innerl~j_t(y^*x)ue(f),ve(g)\innerr\quad(x,y\in\cla_0).
\ee
Using (\ref{fund}), (\ref{ito}) and \eqref{structure relations bdd}  we have:\\
\be       \label{phi ito}
\begin{split}
&\phi_t(x,y)=\int_0^t ds [\phi_s(\theta^0_0(x),y)+\phi_s(x,\theta^0_0(y))+\phi_s(\theta^i_0(x),\theta^i_0(y))\\
&+g^i(s)\phi_s(\theta^i_0(x),y)+g_i(s)\phi_s(x,\theta^0_i(y))+f_i(s)\phi_s(\theta^0_i(x),y)\\
&+f^i(s)\phi_s(x,\theta^i_0(y))+g^i(s)f_j(s)\phi_s(\theta^i_j(x),y)+g_i(s)f^j(s)\phi_s(x,\theta^j_i(y))\\
&+g^i(s)\phi_s(\theta^m_0(x),\theta^m_i(y))+f_i(s)\phi_s(\theta^m_i(x),\theta^m_0(y))+f_j(s)g_i(s)\phi_s(\theta^m_j(x),
\theta^m_i(y))].
\end{split}
\ee
Next we follow the ideas indicated in the pages 178-181 in \cite{dgkbs} and define for m,n in $\IN\cup0,$
\be 
\begin{split}
&\phi^{m,n}_t(x,y):=\frac{1}{(m!n!)^{\frac{1}{2}}}[\innerl~j_t(x)uf^{\otimes^m},j_t(y)vg^{\otimes^n}\innerr~-\innerl~j_t(y^*x)uf^{\otimes^m},vg^{\otimes^n}\innerr~]\\
&=\frac{1}{m!n!}\frac{\partial^m}{\partial\rho^m}\frac{\partial^n}{\partial\eta^n}\{\innerl j_t(x)ue(\rho f),j_t(y)ve(\eta g)  \innerr-\innerl j_t(y^\ast x)ue(\rho f), ve(\eta g) \innerr\}|_{\rho,\eta=0}.
\end{split}
\ee
Differentiating (\ref{phi ito}) with respect to $\rho$ and $\eta$ and setting $\rho=\eta=0,$ we get a recursive integral relation amongst $\phi^{m,n}_t(x,y)$ as follows:
\be \label{recursion}
\begin{split}
&\phi_t^{m,n}(x,y)=\int_0^t ds [\phi_s^{m,n}(\theta^0_0(x),y)+\phi_s^{m,n}(x,\theta^0_0(y))+\phi_s^{m,n}(\theta^i_0(x),\theta^i_0(y))\\
&+g^i(s)\phi_s^{m,n-1}(\theta^i_0(x),y)+g_i(s)\phi_s^{m,n-1}(x,\theta^0_i(y))\\
&+f_i(s)\phi_s^{m-1,n}(\theta^0_i(x),y)+f^i(s)\phi_s^{m-1,n}(x,\theta^i_0(y))\\
&+g^i(s)f_j(s)\phi_s^{m-1,n-1}(\theta^i_j(x),y)+g_i(s)f^j(s)\phi_s^{m-1,n-1}(x,\theta^j_i(y))\\
&+g^i(s)\phi_s^{m,n-1}(\theta^k_0(x),\theta^k_i(y))+f_i(s)\phi_s^{m-1,n}(\theta^k_i(x),\theta^k_0(y))\\
&+f_j(s)g_i(s)\phi_s^{m-1,n-1}(\theta^k_j(x),\theta^k_i(y))],
\end{split}
\ee
where $\phi_t^{-1,n}(x,y):=\phi_t^{m,-1}(x,y):=0$ for all $m,n$ and $x,y\in \cla_0$.
We set in (\ref{recursion}) $m=n=0$ to get 
\begin{equation*} 
\phi_t^{0,0}(x,y)=\int_0^t ds \{ \phi_s^{0,0}(\theta^0_0(x),y)+\phi^{0,0}_s(x,\theta^0_0(y))+\phi_s^{0,0}(\theta^i_0(x),\theta^i_0(y)) \}.
\end{equation*}

Thus we consider an equation of the form
\be\label{first step}
\phi_t^{m,n}(x,y)=\int_0^t ds [\phi_s^{m,n}(\theta^0_0(x),y)+\phi_s^{m,n}(x,\theta^0_0(y))+\phi_s^{m,n}(\theta^i_0(x),\theta^i_0(y))],
\ee
for $x,y\in\cla_0,m,n\geq0$. Our aim is to show that the hypothesis of Theorem \ref{main theorem} and equation (\ref{first step}) imply that $\phi_t^{m,n}(x,y)=0$. Having achieved this, we can embark on our induction hypothesis as $$\phi_t^{k,l}(x,y)=0~\mbox{for $k+l\leq m+n-1$}.$$ 
Under the induction hypothesis, equation (\ref{recursion}) reduces to an equation of the form:
\be   \label{actual equa}
\phi_t^{m,n}(x,y)=\int_0^t ds [\phi_s^{m,n}(\theta^0_0(x),y)+\phi_s^{m,n}(x,\theta^0_0(y))+\phi_s^{m,n}(\theta^i_0(x),\theta^i_0(y))]
\ee
for $x,y\in\cla_0,$ which is an equation similar to (\ref{first step}), leading to $\phi_t^{m,n}(x,y)=0$ as earlier and this will complete the induction process. Thus it only remains to show that the assumptions of this theorem lead to a trivial solution to an equation of the type (\ref{first step}). Omitting the indices m,n, define a map $\psi_t$ belonging to $Lin(\cla_0\ota\cla_0,\IC)$ by: $$\psi_t(x\ot y):=\phi_t^{m,n}(x,y)\quad(x,y\in\cla_0,m,n\geq0)$$ and extending linearly. Thus equation (\ref{first step}) leads to: 
\be\label{tittu}
\psi_t(X)=\int_0^t ds [\psi_s((\theta^0_0\otimes1+1\otimes\theta^0_0+\textit{}(\theta^i_0\ota\theta^i_0))(X))]\quad(X\in\clf).
\ee
The complete positivity of the map $j_t$ implies that 
\be
\innerl j_t(x)\xi,j_t(x)\xi \innerr\leq\innerl j_t(x^*x)\xi,\xi \innerr
\ee
for $\xi\in h\ot\Gamma$ and hence by {\bf A(4)}, we get that
\be \label{du number term}
\begin{split}
\hskip-5pt|\innerl j_t(x)uf^{\otimes^m},j_t(y)vg^{\otimes^n}\innerr\!|
&\!\leq\!(|\innerl j_t(x^*x)uf^{\otimes^m},uf^{\otimes^m}\innerr\innerl j_t(y^*y)vg^{\otimes^n},vg^{\otimes^n}\innerr|)^{\frac{1}{2}}\\
&\!\leq\!(C(m,m,u,u,t,f,f)C(n,n,v,v,t,g,g))^{\frac{1}{2}}\|x\|_2\|y\|_2\\
&=O(e^{\beta t})\|x\|_2\|y\|_2.
\end{split}
\ee
The assumption {\bf A(4)} and (\ref{du number term}) together yield: 
\be\label{dct}
|\psi_t(X)|\leq O(e^{\beta t})\|X\|_\gamma,~\mbox{for $X\in\clf,$}
\ee 
which proves (by virtue of denseness of $\clf$ in $h\otp h$) that $\psi_t$ extends as a bounded map from $h\otp h$ to $\IC.$ If we let $G=\hat{\cll}+B,$ then for $X\in\clf,$ the equation (\ref{tittu}) becomes: $$\psi_t(X)=\int_0^t\psi_s(G(X)) ds.$$ Equation (\ref{dct}) leads to $\int_0^\infty dt e^{-\lambda t}|\psi_t(X)|<\infty$ for $\lambda\geq\beta$, which implies that 
$$\int_0^\infty dt e^{-\lambda t}\psi_t(X)=\int_0^\infty dt e^{-\lambda t}\int_0^t ds \psi_s(G(X)).$$
An integration by parts leads to
\be \label{sesher kobita}
\begin{split}
\int_0^\infty dt e^{-\lambda t} \psi_t((G-\lambda)(X))=0,~\mbox{for $X\in\clf.$}
\end{split}
\ee
Now recall that $\clf$ is a core for $\hat{\cll}$, and moreover by Lemma \ref{right kind of delta}, $B$ is relatively bounded with respect to $\hat{\cll}$. Thus $\clf$ is a core for $G$ as well. So for $Y\in \text{span}\{\cly\},$ let $\{X_n\in\clf\}_n$ be a sequence such that $G(X_n)$ goes to $G(Y)$. We have \be\label{the last equation}\int_0^\infty dt e^{-\lambda t}\psi_t((G-\lambda)(Y))=0\quad(Y\in\cly),\ee by an application of the dominated convergence theorem. In the notation of Lemma \ref{dg}, let $A:=(\hat{\cll}-\lambda)$, $D:=\cly$ and $T:=B$. The inequality (\ref{B ineq}) implies that we are in the set-up of Lemma \ref{dg}. Thus the denseness of $(G-\lambda)(\text{span}\{\cly\})$ follows by applying Lemma \ref{dg}. Therefore equation \eqref{the last equation} and (\ref{dct}) lead to 
$$\int_0^\infty dt e^{-\lambda t}\psi_t(X)=0\quad(X\in h\otp h, \lambda>\beta).$$  This implies that $\psi_t(X)=0$ for $X\in h\otp h$. In particular we have $\phi_t^{m,n}(x,y)=0$ for $x,y\in\cla_0,$ $t\geq0$ and $m,n\geq0$. Hence the result follows.  

\begin{crlre} \label{realpart}Suppose that the trace $\tau$ on the algebra $\cla$ is finite. Let $(j_t)_{t\geq0}$ be a quantum stochastic flow with the noise space $k_0$ and structure matrix $\Theta=\begin{pmatrix}\cll&\delta^\dagger\\\delta&\sigma\end{pmatrix}$, satisfying {\bf A(1)}--{\bf A(2)}, (i)--(ii) of {\bf A(3)} and either {\bf A(4)} or {\bf A(4)$^\prime$}. Moreover, assume the following condition instead of (iii) in {\bf A(3)}:
\[
\cla_0\subseteq D(\cll_2)\cap D(\cll_2^*).
\]
Then $(j_t)_{t\geq0}$ is a $\ast$-homomorphic flow.
\end{crlre}

\begin{proof}
Define a symmetric form $q(x,y)=-\innerl \cll_2(x),y\innerr-\innerl x,\cll_2(y)\innerr$ for all $x,y\in\cla_0,$ with domain $D(q):=\cla_0.$ This form is non-negative by {\bf A(1)} and (ii) of {\bf A(3)}. Since $q(x,y)=\innerl x,(-\cll_2-\cll_2^*)y\innerr$ $\forall x,y\in D(q),$ we may proceed along the lines of the standard proof for the Friedrich extension (see \cite{reed}, vol-II, page-177), which gives a positive self-adjoint operator $Z$ with $D(q)\subseteq D(Z)$ such that $q(x,y)=\innerl x,Z(y)\innerr.$ Set $C=Z^{\frac{1}{2}}.$ 
We have
\be
\frac{d}{dt}\|T_t(x)\|^2=\innerl \cll_2(T_t(x)),T_t(x)\innerr+\innerl T_t(x),\cll_2(T_t(x))\innerr 
=-\|C(T_t(x))\|^2.
\ee
Now we may proceed as in Lemma \ref{extension of C} and Theorem \ref{main theorem} to conclude the result.
\end{proof}
The following corollary is a straightforward application of Lemma \ref{jevaabehoy} and Theorem \ref{main theorem}.
\begin{crlre}
\label{2000}
Let $(j_t)_{t\geq0}$ be a quantum stochastic flow on $\cla$, with the noise space $k_0$ and structure matrix $\Theta=\begin{pmatrix}\cll&\delta^\dagger\\\delta&\sigma\end{pmatrix}$, satisfying {\bf A(1)}--{\bf A(3)}. Suppose that the associated semigroups $(j_t^{c,d})_{t\geq0}$, $c,d\in k_0$ satisfy the estimate ${\bf A(4)}^\prime$ of Lemma \ref{jevaabehoy} with a suitable total subset $\clw$. Then $(j_t)_{t\geq0}$ is a $\ast$-homomorphic flow.
\end{crlre}

Combining Lemma \ref{qs flow is homomorphic for type-1} with Theorem \ref{main theorem}, we get the following:

\begin{crlre}
Let $\cla$ be a type-I von-Neumann algebra with atomic center with the semifinite trace $\tau$ as in the Lemma \ref{qs flow is homomorphic for type-1}. Suppose that $(j_t)_{t\geq0}$ is a quantum stochastic flow on $\cla$, with the noise space $k_0$ and structure matrix $\Theta=\begin{pmatrix}\cll&\delta^\dagger\\\delta&\sigma\end{pmatrix}$, satisfying {\bf A(1)}--{\bf A(3)}. Then $(j_t)_{t\geq0}$ is a $\ast$-homomorphic flow.
\end{crlre}

\section{Applications}
In this section we give some applications of Theorem \ref{main theorem}, including a strong version of the Trotter product formula for quantum stochastic flows with unbounded coefficients and an extension of a previously known dilation result, as obtained in \cite{lingaraj}, for a class of quantum dynamical semigroup on UHF algebras (\cite{matsui}). Using the strong version of the Trotter product formula, we also give a new construction of Brownian motion on compact Lie groups and a construction of random walk on discrete groups. 

\subsection{ A strong Trotter product formula for quantum stochastic flows with unbounded coefficients}\label{trotter's section}

Throughout this section, we fix a C* or von-Neumann algebra $\cla$ equipped with a faithful, semifinite and lower-semicontinuous trace $\tau$ satisfying the assumption {\bf A(2)}, with a dense $\ast$-subalgebra $\cla_0$  and also two C* or von-Neumann $\ast$-homomorphic quantum stochastic flows $(j_t^{(1)})_{t\geq0}$ and $(j_t^{(2)})_{t\geq0}$, with noise spaces $k_1$ and $k_2$ and structure matrices  $\theta^{(1)}:=\begin{pmatrix}\cll^{(1)}&\delta^{\dagger(1)}\\\delta^{(1)}&\sigma^{(1)}\end{pmatrix}$ and $\theta^{(2)}:=\begin{pmatrix}\cll^{(2)}&\delta^{\dagger(2)}\\\delta^{(2)}&\sigma^{(2)}\end{pmatrix}$ respectively, with the common domain algebra $\cla_0$. 

%

Following the discussions in Section 3 of \cite{lindsaykbs}, we define the Trotter product formula of two quantum stochastic flows.

\subsection*{Definition of Trotter product of quantum stochastic flows}

Let $\Gamma_l:=\Gamma(L^2(\IR_+,k_l)),l=1,2$. Set $\Gamma:=\Gamma_1\ot\Gamma_2$, so that $\Gamma:=\Gamma(L^2(\IR_+,k_1\oplus k_2))$. Let $\Xi_t,t\geq0$ be the map defined in Section \ref{Fundamental martingales}, with $k_0:=k_1\oplus k_2$. For $x\in\cla$, let 
\be
\begin{split}
&\eta_t(x):=\left(\widehat{j}^{(1)}_t\bullet \widehat{j}^{(2)}_t\right)(x),
\end{split}
\ee
where $\bullet$ is the product defined in Proposition \ref{matrix space tensor} and $\widehat{j}^{(l)}_t,l=1,2,t\geq0$ is the map, as given in Definition \ref{cocycle definition}. Take a dyadic partition of the whole real line $\IR$ and for $s<t$, consider the part of the partition in the interval $[s,t]$ as described in the picture below:\\


$$|_{[2^ns]\cdot2^{-n}}--\left[s--\left|_{([2^ns]+1)\cdot 2^{-n}}---------\right|_{[2^nt]\cdot 2^{-n}}--t\right],$$
where $[t]:=~\text{greatest integer}\leq t$ for real $t$ and $n\in\IN$ is sufficiently large. 

\bdfn\label{trotter product defined}
Set
\be\label{product formula}
\begin{split}
\phi^{(n)}_{[s,t]}:=&\Phi_{s,([2^ns]+1)2^{-n}}\bullet
\Phi_{([2^ns]+1)2^{-n},2^{-n}}\bullet\Phi_{([2^ns]+2)2^{-n},2^{-n}}\bullet...\\&\bullet
\Phi_{([2^nt]-1)2^{-n},2^{-n}}\bullet\Phi_{[2^nt]2^{-n},t-[2^nt]2^{-n}},
\end{split}
\ee
where $\Phi_{s,t}:=\Xi_s\circ\eta_t$. Set $\phi_t^{(n)}:=\phi_{[0,t]}^{(n)}.$
The map $\phi_t^{(n)}$ will be called the {\it n-fold Trotter product of the flows $(j_t^{(1)})_{t\geq0}$ and $(j_t^{(2)})_{t\geq0}$.}\\
\edfn

Note that the map $\phi_{[s,t]}^{(n)}$ is a $\ast$-homomorphism for each $n$. 

\begin{thm}\label{stkp}
 
Suppose that  $(j^{(l)}_t)_{t\geq0}$, $l=1,2$ are C*-algebraic flows  satisfying {\bf A(1)}--{\bf A(3)} and the conditions of Lemma \ref{jevaabehoy} with the corresponding total sets $\clw_1,\clw_2$ respectively. Furthermore assume the following:

\begin{enumerate}
\item[(a)]
$\cll^{(1)}_2+\cll^{(2)}_2$ is a pre-generator of a $C_0$ contractive and analytic semigroup in $h$ such that $\cla_0$ is a core for the generator.

\item[(b)]
For each $c_j,d_j$ belonging to $\clw_j~(j=1,2)$,
\[ 
\sum_{j=1}^{2}\left(\cll^{(j)}+\innerl c_j,\delta^{(j)}\innerr+\delta^{\dagger(j)}_{d_j}+\innerl c_j,\sigma_{d_j} \innerr+\innerl c_j,d_j \innerr\right)
\]
is a pre-generator of a $C_0$-semigroup in $\cla$.

%
%
%
\end{enumerate}

Then  $\phi_t^{(n)}(x)$ as in the Definition \ref{trotter product defined} converges in the strong operator topology of $B(h\otimes\Gamma)$ to a quantum stochastic flow $(j_t)_{t\geq0}$ with the noise space $k_1\oplus k_2$ and structure matrix 
$$\Theta:=
\left(
\begin{array}{lll}
\mathcal{L}^{(1)}+\mathcal{L}^{(2)} & \delta^{\dagger(1)} &{\delta^{\dagger(2)}}\\
\delta^{(1)} & \sigma^{(1)} &0\\
\delta^{(2)} &0 &\sigma^{(2)}
\end{array}
\right)
.$$

\end{thm}

\begin{proof}

Following the line of arguments given in Section 3 of \cite{lindsaykbs}, it follows that the Trotter product $\phi_t^{(n)}(x)$ for $x\in\cla$, converges in the weak operator topology of $B(h\ot\Gamma)$ to a quantum stochastic flow $(j_t)_{t\geq0}$ with the noise space $k_1\oplus k_2$ and structure matrix $\Theta$.

Let $\Delta_0:=[s,\frac{[2^ns]+1}{2^n})$, $\Delta_j:=[\frac{[2^ns]+j}{2^n},\frac{[2^ns]+j+1}{2^n})$ for $1\leq j\leq [2^nt]-[2^ns]-1$ and $\Delta^\prime:=[\frac{[2^nt]}{2^n},t)$. Let $\chi_0:=1_{\Delta_0}$ $\chi_j:=1_{\Delta_j}$ and $\chi^\prime:=1_{\Delta^\prime}$. Suppose that $M_l\equiv M_l(\|c_l\|,\|d_l\|)$, $l=1,2$ is the constant in the estimate ${\bf A(4)^\prime}$ for $j_t^{(l)}$, i.e. $\| j_t^{(l)c_l,d_l}(x)\|_1 \leq {\rm exp}(tM_l)\| x\|_1$ for $c_l,d_l \in \clw_l$. Then for $c:=c_1\oplus c_2,d:=d_1\oplus d_2$ and $x\in\cla$ we have:
 
\be
\begin{split}
&\innerl e(c1_{[s,t]}),\phi^{(n)}_{[s,t]}(x)e(d1_{[s,t]}) \innerr\\
&=\innerl e(c\chi_0)\bigotimes_{j=1}^{[2^{n}t]-[2^{n}s]-1}  e(c\chi_j)\ot e(c\chi^\prime)~,~\phi_{[s,t]}^{(n)}(x)~e(d\chi_0)\bigotimes_{j=1}^{[2^{n}t]-[2^{n}s]-1}e(d\chi_j)\ot e(d\chi^\prime)\innerr\\
&=\left(\eta^{c,d}_{([2^ns]+1)2^{-n}-s}\circ (\eta^{c,d}_{2^{-n}})^{[2^nt]-[2^ns]-1}\circ\eta^{c,d}_{t-[2^nt]2^{-n}}\right)(x).
\end{split}
\ee
Take $\clw$ to be the total subset of $k_1 \oplus k_2$ consisting of vectors of the form $c_1 \oplus c_2$($c_i \in \clw_i$)  with either $c_1$ or $c_2$ is zero. A direct computation implies that $\eta^{c,d}_w(x)=\left(j_w^{(1)c_1,d_1}\circ j_w^{(2) c_2,d_2}\right)(x).$ Thus for $x\in\cla\cap L^1(\tau)$ and $c,d \in \clw$,
\be
\| \eta^{c,d}_w(x)\|_1
\leq e^{w(M_1+M_2)}\|x\|_1,
\ee
where $M_1,M_2$ depend on $\|c\|,\|d\|$. Thus $$\|\innerl e(c1_{[s,t]}),\phi_{[s,t]}^{(n)}(x)e(d1_{[s,t]})\innerr\|_1\leq e^{(t-s)(M_1+M_2)}\|x\|_1.$$
From this, it follows that the limiting quantum stochastic flow $(j_t)_{t\geq0}$ satisfies
 the hypothesis of Lemma \ref{jevaabehoy}. As $(j_t^{(1)})_{t\geq0}$ and $(j_t^{(2)})_{t\geq0}$ are $\ast$-homomorphic, they satisfy the assumption {\bf A(1)}, from which it follows  by an easy computation that $(j_t)_{t\geq0}$ satisfies {\bf A(1)} too.  The assumption {\bf A(3)} follows from the condition (a) of the present theorem and {\bf A(2)} is  the standing assumption throughout this sub\textit{}section. Thus by Corollary \ref{2000} of Theorem \ref{main theorem}, $(j_t)_{t\geq0}$ is a $\ast$-homomorphic flow. Hence $\phi_{t}^{(n)}(x)$ converges to $j_t(x)$ for each $x\in\cla$ and $t\geq0$ in the strong operator topology of $B(h\ot\Gamma)$.
\end{proof}

In the von-Neumann algebraic case, we have the following theorem for a special class of quantum stochastic flows:

\bthm\label{stkp vNa}
Let the quantum stochastic flow $(j_t^{(l)})_{t\geq0}$, $l=1,2$ be a von-Neumann algebraic flow, satisfying {\bf A(1)}--{\bf A(3)} and the estimate in Lemma \ref{jevaabehoy}. Furthermore assume the following: 

\begin{enumerate}
\item[(a)]
$\sigma^{(j)}=0$ for $j=1,2$.

\item[(b)] The closure of the operator $\mathcal{L}^{(1)}_{2}+\mathcal{L}^{(2)}_{2}$ generates a $C_0$ contractive and analytic semigroup in $h$ such that $\cla_0$ is a core for the generator.
 
\end{enumerate}
Then  $\phi_t^{(n)}(x)$ as in Definition \ref{trotter product defined}, converges in the strong operator topology of $B(h\otimes\Gamma)$ to a quantum stochastic flow $(j_t)_{t\geq0}$ with the noise space $k_1\oplus k_2$ and structure matrix 
$$\Theta:=
\left(
\begin{array}{lll}
\mathcal{L}^{(1)}+\mathcal{L}^{(2)} & \delta^{\dagger(1)} &{\delta^{\dagger(2)}}\\
\delta^{(1)} & 0 &0\\
\delta^{(2)} &0 &0
\end{array}
\right)
.$$

\ethm

\begin{proof}
Set $\theta^{i,(1)}_0(x):=\innerl e_i,\delta^{(1)}\innerr(x),$ $\theta^{i,(2)}_0(x):=\innerl l_i,\delta^{(2)}\innerr(x)$ for $i\geq1$, where $\{e_i\}_i$ and $\{l_i\}_i$ are orthonormal bases for $k_1$ and $k_2$ respectively. 
For $x\in\cla_0$ and every positive integer n, we have

\be
\begin{split}
\|\delta^{(j)}(x)\|^2_2=\sum_i\|\theta^{i,(j)}_0(x)\|^2_2&\leq 2\|\cll^{(j)}_2(x)\|_2\|x\|_2\\
&\leq2\frac{1}{\sqrt{2}~n}\|\cll^{(j)}_2(x)\|_2\frac{1}{\sqrt{2}}n\|x\|_2\\
&\leq\left\{\frac{1}{\sqrt{2}}(\frac{1}{n}\|(\cll^{(j)}_2(x))\|_2+n\|x\|_2)\right\}^2\\
&\mbox{for $j=1,2$ .}\\%
\end{split}
\ee

Thus the operators $\theta^{i,(j)}_0$ are relatively bounded with respect to $\cll^{(j)}_2$ with the bound less than $1$. Similar calculations hold for $\theta^{0,(j)}_i(x)(\equiv\theta^{i,(j)}_0(x^*)^*)$, $j=1,2$. Since $\cll_2^{(j)}$ are the pre-generators of contractive analytic semigroups in $h$, we see that the operators $$\theta^{i,(j)}_0+\theta^{0,(j)}_k+\cll^{(j)}_2,$$ for $j=1,2,i,k\geq1$, are pre-generators of $C_0$-semigroups (see \cite{kato} Theorem 2.4 and Corollary 2.5, p 497-498). This implies that for $c,d\in\clg$ where $\clg:=\{(e_i,0),(0,l_j):~i,j\geq1\}$, $j^{(l)c,d}_t(x)\in\cla\cap h$, for $t\geq0$ and $x\in\cla_0$. Moreover we have:
for $x\in\cla_0$, $c,d\in\clg$, $$\|\innerl~c,\delta^{(1)}\oplus\delta^{(2)}\innerr~(x)\|_2\leq\left\{(\frac{1}{\sqrt{2}~n}\|(\cll^{(1)}_2+\cll_2^{(2)})(x)\|_2)+\frac{n}{\sqrt{2}}\|x\|_2\right\}\|c\|,$$ $$\|(\delta^{(1)}\oplus\delta^{(2)})^\dagger_d(x)\|_2\leq\left\{(\frac{1}{\sqrt{2}~n}\|(\cll^{(1)}_2+\cll_2^{(2)})(x)\|_2)+\frac{n}{\sqrt{2}}\|x\|_2\right\}\|d\|.$$ Thus by hypothesis (b), for $c,d\in\clg$, the operator $$\innerl~c,\delta^{(1)}\oplus\delta^{(2)}\innerr~+(\delta^{(1)}\oplus\delta^{(2)})^\dagger_d+\cll^{(1)}_2+\cll^{(2)}_2+\innerl c,d\innerr$$ generates a $C_0$-semigroup in $h$.  

Let $u,v\in L^\infty(\tau)\cap L^2(\tau)$ and $x\in\cla_0\subset\cla$. We have:
\[
\lgl u,(j^{(1)c,d}_{\frac{t}{n}}\circ j^{(2)c,d}_{\frac{t}{n}})^n(x)v \rgl
=\lgl v^*u, (j^{(1)c,d}_{\frac{t}{n}}\circ j^{(2)c,d}_{\frac{t}{n}})^n(x)\rgl\quad(x\in\cla_0). 
\]
Note that if we let $n\rightarrow\infty$, the term on the right hand side of the above expression converges, by our previous discussions. From this, it follows that $\phi^{(n)}_t(x)$ converges in the weak operator topology of $B(h\ot\Gamma)$, for $x\in\cla_0$ and hence for $x\in\cla$ by the density of $\cla_0$ in $\cla$. The strong convergence follows by an application of Corollary \ref{2000} of Theorem \ref{main theorem} as in the proof of Theorem \ref{stkp}.
\end{proof}

\begin{rmrk}\label{alternative condition}
The conclusions of Theorem \ref{stkp} and Theorem \ref{stkp vNa} also hold under the weaker assumption {\bf A(4)} replacing the estimate of Lemma \ref{jevaabehoy}.
\end{rmrk}

\begin{rmrk}\label{multi-dimensional}
By repeating the above arguments, the conclusions of Theorem \ref{stkp} and Theorem \ref{stkp vNa}  hold for finitely many quantum stochastic flows, each satisfying the corresponding hypotheses.
\end{rmrk}

\subsection{Construction of classical and non-commutative stochastic processes:}

We shall now illustrate how to construct various multidimensional processes as random Trotter product limits of the corresponding ``marginals". Our examples will include Brownian motion on compact Lie groups and random walk on discrete groups. We begin with some general facts about stochastic processes on locally compact groups.\\

Let $G$ be a second countable locally compact group. It is well-known that the topology of such a group is metrizable by a metric which makes it a Polish (complete and separable) space. A choice of such a metric is given by 
\be \rho(g,g^\prime):=\sum_{n=1}^\infty\{\frac{|\phi_n(g)-\phi_n(g^\prime)|}{2^n(1+|\phi_n(g)-\phi_n(g^\prime)|)}\},
\ee
where $\{\phi_i\}_{i=1}^\infty$ is a countable family of functions from $C_0(G)$ which separates points of $G$.

\begin{lmma}\label{asol2}
Let $(X_n)_n$ be a $G$-valued random variable on some probability space $(\Omega,\clf,P)$ and suppose that for all $\psi$ in $L^2(G)$ and for all $\phi$ in $C_0(G),$ 
$$\int_G dg \int_\Omega d\omega |\psi(g)|^2|\phi(g.X_n)-\phi(g.X_m)|^2\longrightarrow0$$ as $n,m\longrightarrow\infty,$ where $dg$ is the left-invariant Haar measure on $G.$ Then there exists a random variable $X:\Omega\longrightarrow G$ such that $X_n\longrightarrow X$ in probability.
\end{lmma}

\begin{proof}
We choose and fix some  $\psi$ in $L^2(G)$ with $\|\psi\|_2=1,$ and let $d\IP(\omega,g):=dP(\omega)\ot|\psi(g)|^2 dg.$ Since $\int_G dg|\psi(g)|^2\int_\Omega dP(\omega) |\phi_i(g.X_n)-\phi_i(g.X_m)|^2\rightarrow0,$ for every $i,$ it follows by setting $Y_n(g,\omega):=g.X_n(\omega),$ and using the dominated convergence theorem that for every $\epsilon>0,$ 

$$\IP(\rho(Y_n,Y_m)\geq\epsilon)\leq\frac{1}{\epsilon^2}\IE^\IP(\rho(Y_n,Y_m))\rightarrow0,$$
as $m,n\longrightarrow\infty$, where $\IE^\IP$ denotes the expectation with respect to the probability measure $\IP$.  Thus there is a $G$-valued random variable Y defined on $\Omega,$ such that 
$$Y_n\stackrel{\IP}{\longrightarrow}Y.$$ So 

$$\IP(\rho(Y_n,Y)\geq\epsilon)=\int_G dg |\psi(g)|^2P(\rho(g.X_n,Y)\geq\epsilon)\equiv\int_G dg |\psi(g)|^2 f_n(g)\longrightarrow0,$$ where $f_n(g):=P(\rho(g.X_n,Y)\geq\epsilon).$ By Egoroff's theorem, there exists a measurable set say $\Delta$ of positive Haar-measure such that for all g in $\Delta,$ we have $f_n(g)\longrightarrow0$ and the proof of the theorem is complete by taking $X(\omega)=g^{-1}Y(g,\omega),$  for any fixed g in $\Delta$.
\end{proof}

We now give a few concrete examples.

\begin{enumerate}
\item[{\bf (A)}] {\bf Classical and non-commutative Brownian motion:}
Assume, as in the Subsection 3.2 (Lemma \ref{asol1} and the discussion following it), that $G$ be a second countable and compact Lie group of dimension k, acting smoothly on a C*-algebra $\cla$ and $\tau$ is a lower-semicontinuous, faithful and finite trace on $\cla$.  Let $\cla^\infty$, $U_g$ and $\{\chi_\ell\}_{\ell=1}^k$ be as in the Subsection 3.2 and let $G_\ell$  denote  the one-parameter subgroup $\exp(t\chi_l),~t\in\IR.$  Define 
$j_t^{(\ell)}:\cla\rightarrow\cla^{\prime\prime}\ot B(L^2(W^{(l)}))(\cong\cla^{\prime\prime}\ot B(\Gamma(L^2(\IR_+)))),$ by  $j_t^{(\ell)}(x)(\omega):=\alpha_{\exp(W^{(\ell)}_t(\omega)\chi_l)}(x),$ where $W^{(\ell)}_t$ is the standard Brownian motion on $\IR.$ It follows from the discussion of Subsection 3.2,  replacing $G$ by $G_\ell,$ that each $j_t^{(\ell)}$ satisfies the assumptions {\bf A(1)}--{\bf A(3)}. Thus, to verify the hypotheses of  Theorem \ref{stkp} for the flows $(j_t^{(l)})_{t\geq0}$, $l=1,2,\ldots,k$, we need to check only the conditions (a) and (b) of that theorem. To this end, let us first verify  that $\cll=\sum_{\ell=1}^k\cll_2^{(\ell)}$ is the pre-generator of a $C_0$-semigroup. For this we proceed as follows:

Let $\delta_\ell$ be the norm generator of the automorphism group $(\alpha_{\exp(t\chi_\ell)})_{t\in\IR}.$ Then $\delta_\ell$ extends to an unbounded, densely defined skew-adjoint operator in $L^2(\tau)$, which generates the unitary group $(U_{\exp(t\chi_\ell)})_{t\in\IR}$ where $\alpha_g$ is implemented by the group of unitaries $U_g$  as discussed in Subsection 3.2.  By an abuse of notation, we again denote this extension by $\delta_\ell.$ Note that $\cll^{(\ell)}_2=\frac{1}{2}\delta_\ell^2$ on $\cla^\infty$ for all $\ell$. Thus  $\sum_{\ell=1}^k\cll^{(\ell)}_2=\frac{1}{2}\sum_{\ell=1}^k\delta_\ell^2$ is a densely defined, negative and symmetric operator. By Nelson's analytic vector theorem for the representations of the Lie algebra \cite[Theorem 3, p. 591]{nelson} and \cite[Theorem X.39]{reed}, $\sum_{\ell=1}^k\cll^{(\ell)}_2$ is essentially self-adjoint in $L^2(\tau)$ and hence its closure generates a $C_0$ contraction and analytic semigroup. This verifies condition (a) of Theorem \ref{stkp}. The 
condition (b) follows by noting  that $\sigma^{(\ell)}=0$ in this case and each of  the maps $\delta^{(l)} \equiv \delta_{\ell}$ is relatively compact with respect to $\cll$. So Theorem \ref{stkp} can be applied  and the limiting flow gives a $\ast$-homomorphic quantum stochastic flow on $\cla$ with generator $\sum_{l=1}^{k}\cll^{(l)}$.  Specializing this to the case when $\cla=C(G)$,  also applying  Lemma \ref{asol2}, we get the convergence in probability of the following sequence of random variables:    

\be
\begin{split}
X_t^{(n)}&:=\prod_{i=1}^k \prod_{l=0}^{[2^nt]}\exp((W^{(i)}_{\frac{[2^nl]+1}{2^n}}-W^{(i)}_{\frac{l}{2^n}})\chi_i),
\end{split}
\ee
where the limiting random variable is clearly a Brownian motion on $G,$ giving a result similar to that in \cite{tkp}.

\begin{paragraph}~
In case of $G=\IT^2$ and $\cla$ the  irrational-rotational C*-algebra $\cla_\theta$ (see page 254 of \cite{dgkbs}), the quantum Brownian motion described in page-275 of \cite{dgkbs} can be constructed using the method described here.
\end{paragraph}

\item[{\bf (B)}] {\bf Random walk on discrete group:} Let G be a discrete and finitely generated group, generated by a symmetric set of torsion free generators, say $\{g_1,g_2,\ldots,g_{2k}\}$. Let e be the identity element of G, $g_1g_{k+1}=e$ and $g\mapsto\alpha_g$ for each $g$ be the automorphism obtained  by the action of $G$ on itself.  Take $\cla=C_0(G)$ and $\tau$ to be the trace with respect to the counting measure. Consider 2k mutually independent Poisson-processes  
$(N^{(i)}_t)_{t\geq0},~i=1,\ldots, 2k,$  on $\IN\cup\{0\},$ with intensity parameter $(\lambda _i)_{i=1}^{2k}$ respectively. Let $Z_t^{(i)}:=N_t^{(i)}-N^{(k+i)}_t,$ $i=1,2,\ldots,k.$  Define 
$j^{(l)}_t:\cla\rightarrow L^\infty(G)\ot B(L^2(N_t^{(l)},N^{(k+l)}_t)),$ $(l=1,2,\ldots,k)$, by 
$j_t^{(l)}(\phi)(\omega)=\alpha_{Z^{(l)}_t(\omega)}(\phi).$ Since the generator $\cll^{(l)}$ of the vacuum expectation  semigroup associated with $j_t^{(l)}$ is bounded, so are the other structure maps. This implies that all the hypotheses of Theorem \ref{stkp} are satisfied. So by Lemma \ref{asol2}, we have convergence in probability of the following sequence of random variables $$X_t^{(n)}:=\prod_{l=0}^{[2^nt]}\prod_{i=1}^{k}\clg^{(i)}_{\frac{l+1}{2^n}}(\clg^{(i)}_{\frac{l}{2^n}})^{-1},$$ where 
$\clg^{(l)}_t(\omega):=g_l^{Z_t^{(l)}(\omega)}.$ The limit is a random variable $X_t$ which is a time homogeneous continuous time simple random walk. 
\end{enumerate}

\subsection{An Application to a class of  stochastic processes on a UHF algebra}

We begin this subsection with the following definition:

\bdfn\label{definition of dilation}
Let $\cla$ be a C* or von-Neumann algebra and $(T_t)_{t\geq0}$ be a quantum dynamical semigroup on $\cla$. A $\ast$-homomorphic C*-algebraic ( or von-Neumann algebraic) quantum stochastic flow $(j_t)_{t\geq0}$ on $\cla$ with domain algebra $\cla_0$ and noise space $k_0$ is called a quantum stochastic dilation of $(T_t)_{t\geq0}$ if we have:

\begin{enumerate}
\item[(i)]
$\cla_0\subseteq D(\cll)$, where $\cll$ is the generator of $(T_t)_{t\geq0}$.
\item[(ii)]
\[
j_t^{0,0}(x)=T_t(x)\quad(x\in\cla,t\geq0). 
\] 
\end{enumerate}
\edfn

In this subsection, we apply the results obtained in Section \ref{trotter's section} to construct quantum stochastic dilation (in the sense of Definition \ref{definition of dilation}) of a class of quantum dynamical semigroups on a uniformly hyperfinite C*-algebra. Let $\cla$ be such an algebra, generated by the infinite tensor product of finite dimensional matrix algebras $M_N(\IC),$ i.e. the C*-completion of the algebraic tensor product $\otimes_{j\in\IZ^d}M_N(\IC)$ where $N$ and $d$ are two fixed positive integers. The unique normalized trace $tr$ on $\cla$ is given by $tr(x)=\frac{1}{N^n}Tr(x)$ for $x\in M_{N^n}(\IC).$ For a simple tensor $a$ $\in\cla$, let $a_{(j)}$ be the $j^{th}$ component of $a$. We define support of $a$, denoted by $supp(a)$ as:
$$supp(a):=\{j\in\IZ^d|~a_{(j)}\neq1\}.$$ 
%
%
 Let $\cla_{loc}$ be the $*$-algebra generated by finitely supported simple tensors in $\cla.$ Clearly $\cla_{loc}$ is dense in $\cla.$  For $k\in\IZ^d,$ the translation $\tau_k$ on $\cla$ is an automorphism determined by $\tau_k(x_{(j)})=x_{(j+k)}$.\\

Note that $M_N(\IC)$ is generated by a pair of non-commutative representatives of the finite discrete group $\IZ_N=\{0,1,2,\ldots,N-1\}$ such that $U^N=V^N=1\in M_N(\IC)$ and $UV=\omega VU$ where $\omega$ is the $N^{th}$ root of unity. Using this fact we get a unitary representation of $\clg=\prod_{j\in\IZ^d}G$ for $G=\IZ_N\times\IZ_N$ in the Hilbert space $L^2(tr)$ as follows: 
$$\clg\ni g\rightarrow U_g=\prod_{j\in\IZ^d}U^{(j)^{\alpha_j}}V^{(j)^{\beta_j}}\in\cla,$$ for $g=\prod_{j\in\IZ^d}(\alpha_j,\beta_j),$ $\alpha_j,\beta_j\in\IZ_N.$ We set $|g|:=\{j\in\IZ_d:~(\alpha_j,\beta_j)\neq(0,0)\}$. For a given completely positive map $\psi$ on $\cla$, formally we define the Linbladian $\cll:=\sum_{k\in\IZ^d}\cll_k$ where for $x\in\cla_{loc}$, $\cll_k(x):=\tau_k\cll_0(\tau_{-k}x)$ with\\ $\cll_0(x)\!:=\!-\frac{1}{2}\{\psi(1)x+x\psi(1)\}+\psi(x)$.  
Consider the Linbladian $\cll$ for the completely positive map $\psi(x):=\sum_{l=1}^{p}r^{(l)*}xr^{(l)},~x\in\cla$, where  $r^{(l)}\in\cla,l=1,2,\ldots,p$. For $g^{(j)}=(\alpha_j,\beta_j)\in G,$ $j\in\IZ^d,$ we set $W_{j,g^{(j)}}:=U^{(j)\alpha_j}V^{(j)\beta_j}\in\cla_{loc}$ and define a set $\clc^1(\cla)$ as: $$\clc^1(\cla)\!:=\!\{x\in\cla:\sum_{j,g}\|W_{j,g^{(j)}}xW_{j,g^{(j)}}^*-x\|\!<\!\infty\}.$$ Matsui (\cite{matsui}) proved that the Linbladian $\cll$ thus constructed is well-defined on $\clc^1(\cla)$, closable and a pre-generator of a quantum dynamical semigroup (denoted by $(T_t^\psi)_{t\geq0}$) on $\cla$. Furthermore $\clc^1(\cla)$ is left invariant by $T_t^\psi$ for each $t\geq0$.

In \cite{lingaraj} the authors considered the problem of constructing quantum stochastic dilations of quantum dynamical semigroups arising as above. However they could construct quantum stochastic dilations for only a special class of semigroups, namely those for which the associated completely positive map $\psi$ is of the form: $\psi(x)=r^*xr$ for $r:=\sum_{g\in\prod_{j\in\IZ^d}\IZ_N}c_gW_g$ with $W_g:=\prod_{j\in\IZ^d}(U^aV^b)^{\alpha_j}$, where $g:=\prod_{j\in\IZ^d}\alpha_j$ such that $\sum_{g}|c_g||g|^2<\infty$, $a,b\in\IZ_N$ being fixed. Our aim is to generalize this result. First of all we will construct dilations by considering $\psi$ of the form given earlier i.e. $\psi(x):=\sum_{m=1}^p r^{(m)*}xr^{(m)}$, each $r^{(m)}$ satisfying either {\bf (I)} or {\bf (II)}  below:

\begin{enumerate}
\item[{\bf (I)}] 
 $r^{(m)}$ is a normal element of $\cla$;

\item[{\bf (II)}] 
 $r^{(m)}\in\cla_{loc}$ with $[r^{(m)},r^{(m)*}]\leq0$.
\end{enumerate}
For each $t\geq0$ set $T_t:=T_t^\psi$ and $T_t^{(m)}:=T_t^{\psi_m}$, where $\psi_m(x):=r^{(m)\ast}xr^{(m)}$, $x\in\cla$. We also denote by $\cll^\psi$ and $\cll^{(m)}$ the norm generators of $(T_t)_{t\geq0}$ and $(T_t^{(m)})_{t\geq0}$ respectively.

Before proving the main theorem of this subsection, we prove a few preparatory lemmas.  
As before, we will view $\cla$ as a C*-subalgebra of $B(L^2(tr))$. 

\begin{lmma}\label{L^1-ext}
Suppose that $(S_t)_{t\geq0}$ is a quantum dynamical semigroup on a C*-algebra $\clb$, equipped with a finite trace $\tau$. Let $\cll$ be the norm generator of $(S_t)_{t\geq0}$ and suppose that $\cla_0\subseteq D(\cll)$ is a dense $\ast$-subalgebra of $\clb$. Let $\tau(\cll(y))\leq0$ for all $y\geq0,~y\in\cla_0$. Then $(S_t)_{t\geq0}$ extends to $L^1(\tau)$ as a contractive $C_0$-semigroup.
\end{lmma}

\begin{proof}
In what follows, $L^1_\IR$ will denote the real Banach space obtained by taking the $L^1$-closure of the real vector space of self adjoint elements of $\clb$, whereas $L^1$ or equivalently $L^1(tr)$ will denote the usual non-commutative complex $L^1$-space of $\clb$.
Let $f(t)=\tau(S_t(y)),~y\geq0$ and $y\in\cla_0$. Then $f^\prime(t)=\tau(\cll(S_t(y)))\leq0$, which implies that $f(t)$ is monotone decreasing. Hence we have $f(t)\leq f(0)$ i.e. 
\be\label{ineq}
\|S_t(y)\|_1\leq\|y\|_1\quad(y\in\cla_0,y\geq0).
\ee
Let $y\in\clb$ be positive so that $y=x^*x$ for some $x\in\clb$. Since $\cla_0$ is dense in $\clb$, there is a sequence $(x_n)_n\in\cla_0$ such that $x_n\rightarrow x$ in $\|\cdot\|_\infty,$  hence  $x_n^*x_n\rightarrow x^*x$ in $\|\cdot\|_\infty.$ As $\tau$ is finite we have $\|\cdot\|_1\leq\|\cdot\|_\infty$. Thus we have $x_n^*x_n\rightarrow x^*x$ in $\|\cdot\|_1$, which implies that the inequality (\ref{ineq}) extends to all the positive elements of $\clb$. Decomposing a general self-adjoint $x$ in $\clb$  as $x=x_+-x_-$ such that $|x|=x_+ + x_-$ with  $x_+\geq0,x_-\geq0$, we have
\be
\|S_t(x)\|_1=\|S_t(x_+)-S_t(x_-)\|_1\leq\|S_t(x_+)\|_1+\|S_t(x_-)\|_1\leq\|x_+\|_1+\|x_-\|_1=\|x\|_1.
\ee
This allows us to extend  $S_t$  as a contractive map on the real Banach space $L^1_\IR.$ We denote this extension by $S_t^{\rm{sa}}$ and consider its complexification $S_t^\prime:L^1\rightarrow L^1$ given by $S_t^\prime(x)=S_t^{\rm{sa}}(Re(x))+iS_t^{\rm{sa}}(Im(x))$, where $x\in L^1\cap\clb$ and $Re(x),$  $Im(x)$ are the real and imaginary parts of $x$ respectively. Clearly, $S^\prime_t|_{\clb}=S_t .$ As $\|Re(x)\|_1\leq\|x\|_1$ and $\|Im(x)\|_1\leq\|x\|_1$, $S_t^\prime$ is a bounded (not necessarily contractive) map on $L^1.$ It follows that the dual map $S_t^{\prime,\ast}:L^\infty\rightarrow L^\infty$ is a weak-$\ast$ continuous map (i.e. ultraweakly continuous in this case). Moreover, observe that for a fixed positive $x\in\clb$ and an arbitrary positive $y\in\clb\cap L^1=\clb$, we have: 
\be
\tau(S_t^{\prime,\ast}(x)y)=\tau(xS_t^\prime(y))=\tau(xS_t(y))\geq0\quad(\text{as}~S_t(y)\geq0~\forall y\geq0),
\ee
hence $S_t^{\prime,\ast}(x)\geq0,$ i.e. $S_t^{\prime,\ast}$ is a positive map. Thus we have: 
\be
\begin{split}
\|S_t^{\prime,*}\|=\|S_t^{\prime,*}(1)\|_\infty=\|S_t^{\rm{sa},*}(1)\|_\infty&=sup_{\|\rho\|_1\leq1,\rho\in L^1_\IR}|\tau(S_t^{\rm{sa},*}(1)\rho)|\\
&=sup_{\|\rho\|_1\leq1,\rho\in L^1_\IR}|\tau(S_t^{\rm{sa}}(\rho))|\leq1.
\end{split}
\ee

Thus for each $t\geq0$, $S_t^{\prime,*}$ is a contractive map on $L^\infty$ and hence its predual map $S_t^{\prime}$ is a contractive map on $L^1$.  The semigroup property as well as the $C_0$ property of $(S^\prime_t)_{t\geq0}$ on $L^1$ follows from the similar properties of $S_t=S_t^\prime|_{\clb}$ with respect to $\|\cdot\|_\infty$ and the fact that $\|\cdot\|_1\leq\|\cdot\|_\infty$.
\end{proof}

%
%
 

\begin{lmma}\label{extends to l2 here as well}
Each of the semigroups $(T_t)_{t\geq0}$ and $(T_t^{(m)})_{t\geq0}$ ($m =1,\ldots,p$) extends to $C_0$-semigroup on $L^2(tr)$.
\end{lmma}

\begin{proof}
For simplicity we will drop the index $m$ and write $r$ for $r^{(m)}$ and $\cll$ for $\cll^{(m)}$. Let $r_k:=\tau_k(r)$ and $\cla_0:=\clc^1(\cla).$ Suppose that $y\in\cla_0$ and $y\geq0.$ A simple computation using assumption {\bf (I)} or {\bf (II)} yields $tr(\cll(y))\leq0.$ Thus by Lemma \ref{L^1-ext}, $(T_t^{(m)})_{t\geq0}$ extends to $L^1$ as a $C_0$ contractive semigroup. 

Let $x\in\cla_0$. Then using contractivity and complete positivity of $T_t^{(m)}$ for each $m$, we have: 
\be
\|T^{(m)}_t(x)\|_2^2=tr(T^{(m)}_t(x^*)T^{(m)}_t(x))\leq tr(T^{(m)}_t(x^*x))=\| T^{(m)}_t(x^*x)\|_1 \leq tr(x^*x)=\|x\|_2^2,
\ee 
 hence each $T_t^{(m)}$ extends to $L^2(tr)$ as a contractive map. The semigroup property and the strong continuity in the $L^2$-norm 
  now follows by the density of  $\cla_0$  in $L^2(tr)$ and the inequality $\|\cdot\|_2\leq\|\cdot\|_\infty$. 

Similar arguments will work for the semigroup $(T_t)_{t\geq0}$.
\end{proof}

\begin{lmma}\label{realpart-2}
Let $\cll_2$ denotes the $L^2$-generator of any of the semigroups $(T_t)_{t\geq0}$ or $(T_t^{(m)})_{t\geq0}$, $m=1,2,\ldots,p$. Then we have:
\[
\clc^1(\cla)\subseteq D(\cll_2)\cap D(\cll_2^*). 
\]
\end{lmma}

\begin{proof}
For simplicity, we drop the superscript $m$ and write $r$ for $r^{(m)}$. Let $r_k:=\tau_k(r)$. We see that formally $\cll_2^{*}(x)=\frac{1}{2}\sum_{k\in\IZ^d}\{r_k[x,r_k^*]+[r_k,x]r_k^*+x[r_k,r_k^*]+[r_k,r_k^*]x\}$. For $x\in \clc^1(\cla)$ we have:
\be\label{L^*}
\hskip-20pt\|\cll_2^*(x)\|\leq\frac{\|r\|}{2}\sum_{k\in\IZ^d}\{\|\delta_k^\dagger(x)\|+\|\delta_k(x)\|\}+\|x\|_\infty\|\sum_{k\in\IZ^d}[r_k,r_k^*]\|,
\ee
where $\delta_k(x):=[x,r_k]$, $x\in\cla$, $k=1,2,\ldots,\infty$. But  $\sum_{k\in\IZ^d}\{\|\delta_k^\dagger(x)\|+\|\delta_k(x)\|\}<\infty$ (see \cite{lingaraj} and \cite{matsui}) and thus it suffices to show the convergence of the third series in (\ref{L^*}). Note that if $r$ satisfies the assumption 
{\bf (I)}, then $[r_k,r_k^*]=0$ for all $k=1,2,\ldots,\infty$ which proves the required convergence. Suppose now  that $r$ satisfies the  assumption {\bf (II)}. As $r=\sum_{g\in\clg}c_g\prod_{j\in\IZ^d}U^{(j)\alpha_j}V^{(j)\beta_j},$ 
\be
\begin{split}
\sum_{k\in\IZ^d}\tau_k\{[r,r^*]\}
=&\sum_{k\in\IZ^d}\sum_{g,h\in\clg}c_g\overline{c_h}\tau_k\{[\prod_{j\in\IZ^d}U^{(j)\alpha_j}V^{(j)\beta_j},
\prod_{j\in\IZ^d} V^{(j),-\beta^\prime_j} U^{(j),-\alpha^\prime_j}]\}\\
=&\sum_{k\in\IZ^d}\sum_{g,h\in\clg}c_g\overline{c_h}[\prod_{j\in\IZ^d}U^{(j+k)\alpha_j}V^{(j+k)\beta_j},
\prod_{j\in\IZ^d}V^{(j+k),-\beta^\prime_j}U^{(j+k),-\alpha^\prime_j}]\\
=&\sum_{k\in\IZ^d}\sum_{g,h\in\clg}c_g\overline{c_h}[\prod_{j\in\IZ^d}U^{(j)\alpha_{j-k}}V^{(j)\beta_{j-k}},
\prod_{j\in\IZ^d}V^{(j),-\beta^\prime_{j-k}}U^{(j),-\alpha^\prime_{j-k}}].
\end{split}
\ee
As  $r\in\cla_{loc}$, we have $\alpha_{j-k}=\beta_{j-k}=\alpha^\prime_{j-k}=\beta^\prime_{j-k}=0\in\IZ_N $ for $|k|\geq M$ for sufficiently large integer $M$. Hence $[r_k,r_k^*]=0$ for such $k.$ Thus the series is actually finite and hence $\|\cll_2^*(x)\|<\infty,$ i.e. 
$\cla_0\subseteq D(\cll_2)\cap D(\cll_2^*)$ as required.
\end{proof}

\begin{lmma}\label{asol kaaj}
For each $m=1,2,\ldots,p$, there exists a C*-algebraic quantum stochastic flow $(j_t^{(m)})_{t\geq0}$ on $\cla$ with domain algebra $\cla_0:=\clc^1(\cla)$,  which gives a quantum stochastic dilation of the semigroup $(T_t^{(m)})_{t\geq0}$. Furthermore $(j_t^{(m)})_{t\geq0}$ satisfies assumptions {\bf A(1)}, {\bf A(2)}, (i) and (ii) of {\bf A(3)} and the hypotheses of Lemma \ref{jevaabehoy} and Corollary \ref{realpart}, with $\cla_0=\clc^1(\cla)$.
\end{lmma}

\begin{proof}
The existence of the dilation $(j_t^{(m)})_{t\geq0}$ is the main result in \cite{lingaraj}. Since the flow $(j_t^{(m)})_{t\geq0}$ is a $\ast$-homomorphic flow for each $m=1,2,...,p$, by Lemma \ref{homomorphism implies structure relations} the flow satisfies assumption {\bf A(1)}. On  $\cla$, $tr$ is a faithful and finite trace and $\cla_0$ is a dense $\ast$-subalgebra of $\cla$, which implies {\bf A(2)}. The statement in the first part of (i) of {\bf A(3)} follows from Lemma \ref{extends to l2 here as well}. From \cite{matsui} it follows that $\cla_0$ is a core for $\cll^{(m)}$, as $\cla_0$ is left invariant by the semigroup $(T_t^{(m)})_{t\geq0}$ for each $m=1,2,\ldots,p$. Furthermore, since $\cla_0$ is dense in $L^2(tr)$ and the $L^2$-extension of $(T_t^{(m)})_{t\geq0}$ leaves it invariant for each $m$, it follows that $\cla_0$ is also a core for $\cll_2$. Thus the second part in (i) of {\bf A(3)} follows. The statement in (ii) of {\bf A(3)} follows by a simple computation. From Lemma \ref{realpart-2} it follows that the hypotheses of Corollary \ref{realpart} hold for $(j_t^{(m)})_{t\geq0}$. We prove that it also satisfies the hypothesis of Lemma \ref{jevaabehoy} as follows:

Let $\delta_j^{(m)}(x):=[x,\tau_j(r^{(m)})],$ $x\in\cla_0.$ Then 
\be\label{l1}
\|\delta_j^{(m)}(x)\|_1\leq2\|\tau_j(r^{(m)})\|_\infty\|x\|_1.
\ee
Thus $\delta_j^{(m)}$ extends to a bounded operator on $L^1$ for each $m$.  Similar result holds for $\delta_j^{\dagger,(m)}$. Since $(j_t^{(m)})_{t\geq0}$ is a 
C*-algebraic quantum stochastic flow, it follows from Lemma \ref{all semigroups are c0} that the semigroup $(j_t^{(m)~e_i,e_j})_{t\geq0}$ is $C_0$ and its 
generator restricted to $\cla_0$ is given by $\cll^{(m)}+\delta_i^{(m)}+\delta^{\dagger(m)}_j$, where $\{e_j\}_{j}$ is an orthonormal basis for the associated 
noise space $k_0$ (see \cite{lingaraj}). By Lemma \ref{L^1-ext}, $(T^{(m)}_t)_{t\geq0}$ extends to  a contractive $C_0$-semigroup on $L^1(tr)$. Let us denote the 
generator of the extended semigroup by $\cll^{(m)}_1$. Clearly, as $\cla_0$ is dense in $L^1(tr)$ and is left invariant by the semigroup $(T^{(m)}_t)_{t\geq0}$, it 
is a core for $\cll^{(m)}_1$. As the maps $\delta_i^{(m)}$ and $\delta^{\dagger(m)}_j$ are bounded maps on $L^1(tr)$, it follows that  the map 
$\cll^{(m)}_1+\delta_i^{(m)}+\delta^{\dagger(m)}_j$ generates a $C_0$-semigroup on $L^1(tr)$ and moreover, $\cla_0$ is a core for this map as well. This proves 
that the semigroup $(j_t^{(m)~e_i,e_j})_{t\geq0}$ extends to a $C_0$-semigroup on $L^1(tr)$. Similar conclusion will hold if we replace  $\delta_i,\delta^\dagger_j$ by $z \delta_i, z\delta^\dagger_j$'s for any fixed $z \in \IC$, hence the hypothesis of Lemma \ref{jevaabehoy} is  satisfied with the total set 
$\clw$ being $\{ ze_i;~z\in \IC,~i \geq 1\}$.
\end{proof}

Now we prove the main theorem of this subsection.

\begin{thm} \label{notun-qsde}
Suppose that $(j^{(m)}_t)_{t\geq0}$ is the quantum stochastic flow with structure maps $\cll^{(m)},\delta_i^{(m)},\delta^{\dagger(m)}_i,i=1,2,\ldots,\infty$, as obtained in Lemma \ref{asol kaaj} for $m=1,2,\ldots,p$. Then the quantum stochastic differential equation:
\be
\begin{split}
dj_t(x)=\sum_{j\in\IZ^d}\sum_{m=1}^p j_t(\delta_j^{\dagger~(m)} (x))&da_j^{(m)}(t)+\sum_{j\in\IZ^d}\sum_{m=1}^p j_t(\delta_j^{(m)}(x))da^{\dagger~(m)}_j(t)+\sum_{m=1}^{p}j_t(\cll^{(m)})dt,\\
&j_0(x)=x\ot I_\Gamma,
\end{split}
\ee
where $x\in\cla_0$ and $\Gamma$ is a symmetric Fock space, admits a $\ast$-homomorphic solution $(j_t)_{t\geq0}$ which is a C*-algebraic quantum stochastic flow and gives a quantum stochastic dilation of the quantum dynamical semigroup $(T_t^\psi)_{t\geq0}$.
\end{thm}

\begin{proof}
We will apply Theorem \ref{stkp} to the collection of quantum stochastic flows: $(j^{(m)}_t)_{t\geq0}$, $m=1,2,\ldots,p$ (see Remark \ref{multi-dimensional}). By Lemma \ref{asol kaaj} the flows satisfy assumptions {\bf A(1)}, {\bf A(2)}, (i) and (ii) of {\bf A(3)} and the hypothesis of Corollary \ref{realpart}. By virtue of Remark \ref{alternative condition}, it suffices to verify conditions (b) and the core property of $\cla_0$  as in (a) of Theorem \ref{stkp}. 

Observe that the map $\sum_{m=1}^{p}\cll^{(m)}$ is the pre-generator of the quantum dynamical semigroup $(T_t^\psi)_{t\geq0}$ (by \cite{matsui}). Moreover, the maps $\delta_j^{(m)},\delta_j^{\dagger(m)},m=1,2,\ldots,p,j=1,2,\ldots,\infty$ are bounded maps on $\cla$. Thus the map $\sum^{p}_{m=1}(\cll^{(m)}+\delta_i^{(m)}+\delta^{\dagger(m)}_j)$ generates a $C_0$-semigroup on $\cla$. Thus condition (b) of Theorem \ref{stkp} holds.

From \cite{matsui} it follows that the dense $\ast$-subalgebra $\cla_0:=\clc^1(\cla)$ is a core for the map $\sum_{m=1}^{p}\cll^{(m)}$. 

Hence by Theorem \ref{stkp} we get a quantum stochastic flow $(j_t)_{t\geq0}$ which satisfies the required quantum stochastic differential equation.
\end{proof}

\begin{crlre}
Let $r^{(m)}=\sum_{g\in\IZ_N}c_gW_g$ for each $m,$ where $W_g=\prod_{j\in\IZ^d}(U^aV^b)^{\alpha_j}$ for $g=\prod_{j\in\IZ^d}\alpha_j$ (as in \cite{lingaraj}). Then the hypotheses of Theorem \ref{notun-qsde} are satisfied and the same conclusion follows, which generalizes the dilation result obtained in \cite{lingaraj}.
\end{crlre}

\noindent\textbf{Acknowledgment:}
The second author is partially supported by a project on `Non-commutative Geometry and Quantum Groups' funded by Indian National Science Academy and by Swarnajayanti Fellowship and Grant from DST, Govt. of India. The third author is supported by the Bhatnagar Fellowship of CSIR. All the authors also acknowledge the support of UK-India Education and Research Initiative (UKIERI).

\end{document}